\documentclass[final,12pt]{colt2025} 

\title[The Adaptive Complexity of Finding a Stationary Point]{The Adaptive Complexity of Finding a Stationary Point}
\usepackage{times}
\DeclareSymbolFont{extraup}{U}{zavm}{m}{n}
\DeclareMathSymbol{\varheart}{\mathalpha}{extraup}{86}
\DeclareMathSymbol{\vardiamond}{\mathalpha}{extraup}{87}

\newcommand{\R}{\mathbb{R}}

\def\eps{\varepsilon}

\usepackage[algo2e,ruled]{algorithm2e}
\LinesNumbered

\SetKwComment{Comment}{$\triangleright$\ }{}

\providecommand\norm[1]{\|#1\|}

\makeatletter
\newcommand*{\inlineequation}[2][]{%
  \begingroup
    \refstepcounter{equation}%
    \ifx\\#1\\%
    \else
      \label{#1}%
    \fi
    \relpenalty=10000 %
    \binoppenalty=10000 %
    \ensuremath{%
      #2%
    }%
    ~\@eqnnum
  \endgroup
}
\makeatother

\newtheorem{fact}[theorem]{Fact}

\usepackage{threeparttable}

\usepackage{xspace}

\xspaceaddexceptions{[]\{\}}

\usepackage{hyperref}

\usepackage{thm-restate}

\newcommand{\eat}[1]{}

\makeatletter
\newcommand*{\rom}[1]{\expandafter\@slowromancap\romannumeral #1@}

\usepackage{multirow}

\usepackage{tablefootnote}
\usepackage{float}
\usepackage{makecell}
\usepackage{enumitem}

\allowdisplaybreaks

\DeclareMathSymbol{\varheartsuit}{\mathalpha}{extraup}{86}
\DeclareMathSymbol{\vardiamondsuit}{\mathalpha}{extraup}{87}

\coltauthor{%
 \Name{Huanjian Zhou} \Email{zhou@ms.k.u-tokyo.ac.jp}\\
 \addr University of Tokyo, RIKEN AIP
 \AND
 \Name{Andi Han} \Email{andi.han@sydney.edu.au}\\
 \addr RIKEN AIP, University of Sydney 
 \AND
 \Name{Akiko Takeda} \Email{takeda@mist.i.u-tokyo.ac.jp}\\
 \addr University of Tokyo, RIKEN AIP
 \AND
 \Name{Masashi Sugiyama} \Email{sugi@k.u-tokyo.ac.jp}\\
 \addr RIKEN AIP, University of Tokyo 
}

\begin{document}

\maketitle

\begin{abstract}
In large-scale applications, such as machine learning, it is desirable to design non-convex optimization algorithms with a high degree of parallelization. In this work, we study the adaptive complexity of finding a stationary point, which is the minimal number of sequential rounds required to achieve stationarity given polynomially many queries executed in parallel at each round. 

For the high-dimensional case, \emph{i.e.}, $d = \widetilde{\Omega}(\varepsilon^{-(2 + 2p)/p})$, we show that for any (potentially randomized) algorithm, there exists a function with Lipschitz $p$-th order derivatives such that the algorithm requires at least $\varepsilon^{-(p+1)/p}$ iterations to find an $\varepsilon$-stationary point.  Our lower bounds are tight and show that even with $\mathrm{poly}(d)$ 
queries per iteration, no algorithm has better convergence rate than those achievable with one-query-per-round algorithms. In other words, gradient descent, the cubic-regularized Newton's method, and the $p$-th order adaptive regularization method are adaptively optimal. Our proof relies upon novel analysis with the characterization of the output for the hardness potentials based on a chain-like structure with random partition.

For the constant-dimensional case, \emph{i.e.}, $d = \Theta(1)$, we propose an algorithm that bridges grid search and gradient flow trapping, finding an approximate stationary point in constant iterations. Its asymptotic tightness is verified by a new lower bound on the required queries per iteration. We show there exists a smooth function such that any algorithm running with $\Theta(\log (1/\varepsilon))$ rounds requires at least $\widetilde{\Omega}((1/\varepsilon)^{(d-1)/2})$ queries per round. This lower bound is tight up to a logarithmic factor, and implies that the gradient flow trapping is adaptively optimal.


\end{abstract}


\newpage
\tableofcontents
\newpage

\section{Introduction}
Let $f : \R^d \to \R$ be a smooth function (\emph{i.e.}, the map $\boldsymbol{x} \mapsto \nabla f(\boldsymbol{x})$ is Lipschitz, and $f$ is possibly non-convex). 
The problem of minimizing $f$ is one of the most fundamental problems across a wide range of scientific and engineering disciplines, with particular importance in modern machine learning~\citep{jain2017non,bottou2018optimization}.
Generally, without additional structured assumptions on $f$, it is $\mathcal{NP}$-hard to find approximate global minima or even test if a point is a local minimum~\citep{nemirovskij1983problem} or a high-order saddle point~\citep{murty1985some}.
As an alternative measure of optimization convergence, we study the problem of finding an $\varepsilon$-approximate stationary point (for $\varepsilon > 0$), \emph{i.e.}, a point $\boldsymbol{x}\in \R^d$ such that 
\[\left\|\nabla f(\boldsymbol{x})\right\| \leq \varepsilon.\]
This is motivated by a line of research that identifies sub-classes of non-convex problems for which all (first-order or second-order) stationary points are globally optimal~\citep{choromanska2015loss,ge2016matrix,kawaguchi2016deep,ge2015escaping,sun2018geometric,ma2018implicit,allen2019convergence}. In addition, practically efficient (gradient-based) algorithms have been developed to find stationary points \citep{kingma2014adam,jin2017escape,zaheer2018adaptive,fang2018spider}. 
%
Furthermore, significant progress has been made in  developing sequential algorithms for this problem and proving non-asymptotic convergence rates~\citep{cartis2023scalable,nesterov2006cubic,nesterov2013introductory,birgin2017worst,cartis2020concise,cartis2020sharp}. 

The algorithms underlying the above results are highly sequential and fail to fully exploit contemporary parallel computing resources such as multi-core central processing units (CPUs) and many-core graphics processing units (GPUs).
Nevertheless, given the widespread application of non-convex optimization in machine learning and the continual growth in dataset sizes~\citep{bottou2018optimization}, there is a persistent need to accelerate non-convex optimization through parallelization~\citep{dean2012large,you2017large,recht2011hogwild,You2020Large}.

A convenient metric for parallelization in black-box oracle models is \emph{adaptivity}, which is the number of sequential rounds it makes when each round can execute polynomial independent queries in parallel.
Over the past several years, there have been
breakthroughs in the study of adaptivity in various problems including sorting~\citep{valiant1975parallelism,cole1988parallel,braverman2016parallel}, multi-armed bandits~\citep{agarwal2017learning}, property testing~\citep{canonne2018adaptivity,chen2018settling}, submodular optimization~\citep{balkanski2018adaptive,chakrabarty2024parallel,li2020polynomial}, convex optimization~\citep{balkanski2018parallelization,diakonikolas2019lower,bubeck2019complexity,carmon2023resqueing}, and log-concave sampling~\citep{zhou2024parallel,anari24a,zhou2024adaptive}.
%


While the adaptive complexity of convex optimization has been extensively studied—most notably with \citep{balkanski2018parallelization} demonstrating that parallelization does not provide acceleration for convex optimization—the adaptive complexity of \textit{non-convex} optimization has not been explored.
This gap motivates our investigation into the question:
\begin{center}
\textit{Whether parallelization can fundamentally accelerate non-convex optimization   
in both {constant- and high-dimensional settings}}?
\end{center}
%

\subsection{Our results}

{
\begin{table}[t!]
\footnotesize
\renewcommand\arraystretch{1.4}
\centering
\caption{Comparisons of the state-of-the-art algorithms and our lower bounds. When $d = \Theta(1)$, we focus on the optimization over cube constraint $[0,1]^d$. When $d = \widetilde{\Omega}\left(\varepsilon^{-\frac{2+2p}{p}}\right)$, we consider $p$-th order smoothness and assume access to a $p$-th order oracle. 
Here, $\widetilde{\Omega}$ omits logarithmic factors in the parameter $1/\varepsilon$.
}
\resizebox{0.99\textwidth}{!}{%
\begin{tabular}{|c|c|c|c|}
\hline
\multicolumn{2}{|c|}{{Problem setting}} & 
\makecell[c]{Adaptive complexity}&\makecell[c]{Queries per iteration} \\
\hline
\multirow{4}{*}{ \makecell{\\\\(constant) ~~ $d = \Theta(1)$\\ ($d\geq 2$)}} & \makecell[c]{Upper bounds} &  
$k = \Theta(1)$ 
~(Theorem~\ref{the:main3})
&
$\mathcal{O}\bigg(
\varepsilon^{-\frac{\left({d+1}\right)^k}{\left({2d}\right)^k-(d+1)^k}\cdot \frac{d-1}{2}}
\bigg)$ 
\\
\cline{2-4}
  & \makecell[c]{Lower bounds} &   $k = \Theta(1)$~(Theorem~\ref{the:main2}) & $\widetilde{{\Omega}}\left(\varepsilon^{-\frac{d^k}{d^k-1}\cdot \frac{d-1}{2}}\right)$\\
  \cline{2-4}
  & \makecell[c]{Upper bounds} &  
$\Theta\left(\log \left(\frac{1}{\varepsilon}\right)\right)$ 
~\citep{hollender2023computational} 
&
$\mathcal{O}\left(\varepsilon^{-(d-1)/2}\right)$ 
\\
  \cline{2-4}
  & \makecell[c]{Lower bounds} &   $\Theta(\log (\frac{1}{\varepsilon}))$~(Theorem~\ref{the:main2}) & $\widetilde{{\Omega}}\left(\varepsilon^{-(d-1)/2}\right)$\\
\hline
\multirow{3}{*}{(high-dim.)~ $d = \widetilde{\Omega}\left(\varepsilon^{-\frac{2+2p}{p}}\right)$ }   & \makecell[c]{Upper bounds} & \makecell[c]{$\mathcal{O}\left(\varepsilon^{-\frac{1+p}{p}}\right)$~\citep{birgin2017worst}} & $1$\\
\cline{2-4}
  & \makecell[c]{Lower bounds}  &  ${\Omega}\left(\varepsilon^{-\frac{1+p}{p}}\right)$ \citep{carmon2020lower}      & $1$ \\
\cline{2-4}
  & \makecell[c]{Lower bounds}  &  ${\Omega}\left(\varepsilon^{-\frac{1+p}{p}}\right)$~(Theorem~\ref{the:main})      & $\mathsf{poly}(d)$ \\
\hline
\end{tabular}
}
\vspace{-0.2cm}
\label{table:results}
\end{table}
}

In this paper, we make significant progress by establishing new lower bounds for the high-dimensional setting, as well as both upper and lower bounds in the constant-dimensional case.
Notably, our lower bounds either match the best known upper bounds or match our proposed new upper bounds up to logarithmic factors. As a result, we obtain the first \emph{tight} adaptive complexity characterizations for finding stationary points.

\paragraph{Lower bounds in high dimension.}
We derive the following adaptive complexity in the high-dimensional setting  $d = \widetilde{\Omega}\left(\varepsilon^{-(2+2p)/{p}}\right)$.
\begin{theorem}[\bf{informal, see~Theorem~\ref{the:main}}]
\label{the:inf1}
For $d= \widetilde{\Omega}\left(\varepsilon^{-(2+2p)/{p}}\right)$ and $p$-th order smooth function, any randomized optimizer needs $\Omega(\varepsilon^{-(1+p)/p})$ sequential rounds to find $\varepsilon$-stationary points even with $\mathsf{poly}(d)$ queries per round.
\end{theorem}
%
%
Our result significantly strengthens the lower bound established for one-query-per-round algorithms in \citep{carmon2020lower}, even with only logarithmic scaling in the dimension dependence.
%
%
Specifically, for $p=1$, our lower bound implies that gradient descent~\citep{nesterov2012make} is adaptive optimal among all methods (even randomized, high-order, parallel methods) operating on functions with Lipschitz continuous gradient and bounded initial sub-optimality. 
Similarly, in the case $p = 2$, our result shows that the cubic regularization of Newton's method~\citep{nesterov2006cubic,cartis2010complexity} is adaptive optimal 
and for general $p$, 
the $p$-th order adaptive regularization (ARp) algorithm~\citep{birgin2017worst,cartis2020sharp,cartis2020concise} is adaptive optimal.
Consequently, our result highlights that \textit{parallelization offers no acceleration} for high-dimensional non-convex optimization.
%

%

%
Our result can be also viewed as a specific setting of stochastic non-convex optimization with zero variance, (\emph{i.e.}, $\sigma =0$). 
Combining with the lower bound for stochastic non-convex optimization with $\sigma>0$, $\Omega(\sigma^2\varepsilon^{-4})$~\citep{arjevani2023lower}, we can obtain a lower bound $\Omega(\varepsilon^{-2} + \sigma^2\varepsilon^{-4})$ matching the query complexity by stochastic gradient descent (SGD), \emph{i.e.}, SGD is also adaptive optimal for stochastic non-convex optimization.
%
We also note the construction in \citep{arjevani2023lower} cannot directly apply to our setting
(see more discussions in Section \ref{sect:technical_overview}). 
%

%
%


\paragraph{Lower and upper bounds in constant dimension.}
Following \citep{bubeck2020trap}, we study the constraint set $[0,1]^d$ ($d\geq 2$) for ease of exposition. For constant dimension $d$, the na\"ive grid search only requires a single round by querying $O(\varepsilon^{-d})$ points on an $O(\varepsilon)$-net of $[0,1]^d$. 
The first non-trivial improvement was proposed in~\citep{bubeck2019complexity}, where they developed an algorithm running within $O(\mathsf{poly}(d)\log(\frac{1}{\varepsilon}))$ iterations and issuing $O(\mathsf{poly}(d,\log(\frac{1}{\varepsilon}))\cdot \varepsilon^{-(d-1)/2})$ queries per iteration.  The number of queries per iteration has been later improved to $O(\mathsf{poly}(d)\cdot \varepsilon^{-(d-1)/2})$ by an algorithm called gradient flow parallel trap (GFPT)~\citep{hollender2023computational}. 
In this work, we bridge grid search ($1$-round) and GFPT ($\Theta(\log (\frac{1}{\varepsilon}))$-round) by establishing the following upper bound on the query complexity for constant-round algorithms.
\begin{theorem}[\bf{informal, see Theorem~\ref{the:main3}}]
\label{the:inf2}
For $d=\Theta(1)$ (with $d\geq 2$) and 
 any Lipschitz smooth function, there exists an algorithm running within $k = \Theta(1)$ sequential rounds to find $\varepsilon$-stationary points with $\varepsilon^{-\frac{d-1}{2}(1+\mathcal{O}(2^{-k}))}$ queries per iteration.
\end{theorem}
This is the \textit{first} result for constant-dimension and constant-iteration settings. 
For $k = 1$, our result recovers the $\varepsilon^{-d}$ complexity of grid search.
As $k$ increases, our complexity approaches $\mathcal{O}(\varepsilon^{-\frac{d-1}{2}}) $ exponentially fast in the exponent and becomes numerically close to the complexity of GFPT after a moderate number of iterations. 
Because the query complexity is non-increasing with the number of rounds, for $\omega(1)\leq k\leq \mathcal{O}(\log (\frac{1}{\varepsilon}))$, the number of required queries per round lies in $(\varepsilon^{-\frac{d-1}{2}},\varepsilon^{-\frac{d-1}{2}(1+\mathcal{O}(2^{-k}))})$ which remains $\mathsf{poly}(\frac{1}{\varepsilon})$. 
%
%

We also complement our upper bound in constant rounds by showing new lower bounds for any algorithm running within $\mathcal{O}(\log (\frac{1}{\varepsilon}))$ iterations.
\begin{theorem}[\bf{informal, see Theorem~\ref{the:main2}}]
\label{the:inf3}
For $d= \Theta(1)$ and any Lipschitz smooth function, to find $\varepsilon$-stationary points, any randomized optimizer running within $k = \mathcal{O}(\log (\frac{1}{\varepsilon}))$ needs at least $\widetilde{{\Omega}}\left(\varepsilon^{-\frac{d-1}{2} (1+\mathcal{O}(d^{-k}))}\right)$ queries per round.   
\end{theorem}
%
This lower bound asymptotically matches convergence rate of our algorithm as iteration $k\to \infty$.
%
%
Our established lower bound also matches the upper bound of GFPT~\citep{hollender2023computational} with $\Theta(\log(\frac{1}{\varepsilon}))$ iterations.
Our results partially address the question 3 in \citep{bubeck2020trap} by showing $\mathsf{poly}(\frac{1}{\varepsilon})$ complexity of $\mathcal{O}(\log (\frac{1}{\varepsilon}))$-iterations algorithm, \emph{i.e.}, the low-depth region is at least $\Omega(\log (\frac{1}{\varepsilon}))$.

\subsection{Technical overview}
\label{sect:technical_overview}

In this section we summarize the main technical ideas used to prove our lower bounds.  For details, see Section \ref{sec:high} for Theorem \ref{the:inf1} and Section \ref{sec:constant} for Theorem \ref{the:inf2}. 
Notably, the high-dimensional lower bound relies on concentration inequalities specific to high dimensions, whereas the constant-dimensional lower bound is based on a reduction whose cost remains constant only when the dimension itself is fixed.

\subsubsection{High-dimensional case}
\label{sec:tech1}
Existing works~\citep{woodworth2017lower,carmon2020lower,carmon2021lower,yue2023lower,kwon2024complexity,zhang2022lower,arjevani2023lower} that proved lower bounds for continuous optimization are based on a family of hard functions, parameterized by random orthogonal vectors $(\boldsymbol{u}^1,\ldots,\boldsymbol{u}^r)$, where $r$ is the number of rounds. 
%
A crucial characteristic of such construction is that no algorithm can access information about the remaining vectors $\boldsymbol{u}^{i+1},\ldots,\boldsymbol{u}^r$ at round $i$, with high probability.
Moreover, obtaining an approximate stationary point critically requires the knowledge of the final vector $\boldsymbol{u}^r$.
%
%

To derive our adaptive lower bounds, we follow a similar chaining-like construction. Nonetheless, compared to prior works, we design a family of hard functions  against a \emph{polynomial} number of parallel queries, which is the main technical innovation of this paper. 
Our new construction is based on \emph{random partition} used in exploring adaptive complexity of log-concave sampling and submodular optimization~\citep{zhou2024adaptive,chakrabarty2022improved,li2020polynomial}.
Such random partition based chaining structure allows us to simplify the proof and elegantly establish the adaptive complexity for non-convex optimization.

\paragraph{Prior arguement fails under polynomial queries.}
We first revisit the construction in \citep{carmon2020lower} and highlight why their approach fails to address cases involving a polynomial number of queries.
Inspired by chaining-like quadratic functions for smooth convex optimization~\citep{nesterov2013introductory}, their work designed a robust chaining-like function parameterized by an orthogonal matrix $\boldsymbol U= (\boldsymbol{u}^1,\ldots,\boldsymbol{u}^r)$
such that (i) $\left\|\nabla f(\boldsymbol U\boldsymbol{x})\right\|$ is non-vanishing unless $|\langle \boldsymbol{u}^i,\boldsymbol{x}\rangle|>1$ for all $i\in \{ 1,2,...,r\}$; 
(ii) the gradient   $\nabla f(\boldsymbol U\boldsymbol{x})$ will be independent of $\boldsymbol{u}^{i+1},\ldots,\boldsymbol{u}^r$ if $|\langle \boldsymbol{u}^j,\boldsymbol{x}\rangle|<1/2$, for $j\geq i+1$; (iii) without $\boldsymbol{u}^i$ being discovered, $|\langle \boldsymbol{u}^i,\boldsymbol{x}\rangle|<1/2$ with high probability.
Thus, any algorithms can discover at most one random vector in each iteration, and thus cannot reach an approximate stationary point without iterating $r$ times.

Properties (i) and (ii) rely on the inherent structure of the chaining-like function, while the remaining objective is to establish property (iii) using random orthogonal vectors.
To prove (iii), they analyzed the behavior at iteration $t$. Let $\widehat{\boldsymbol{u}}^j$ denote the projection of the future vector $\boldsymbol{u}^j$ with $j>t$ onto the orthogonal complement of the space $\mathrm{span}\{\boldsymbol{u}^1, \ldots, \boldsymbol{u}^t, \boldsymbol{x}^1, \ldots, \boldsymbol{x}^t\}$, where $\boldsymbol{x}^1, ..., \boldsymbol{x}^t$ denote the iterates up to $t$-th iteration. 
They showed that, conditioned on $\boldsymbol{u}^1, \ldots, \boldsymbol{u}^t$, the random variable $\widehat{\boldsymbol{u}}^j$ follows a rotationally symmetric distribution and is independent of $\boldsymbol{x}^{t+1}$.  
Using a subspace concentration and an inductive argument \citep{ball1997elementary}, they verified that property (iii) holds.
However, this relies on the projection, and when a set of queries is issued as a basis for the entire space (which can be done with polynomial queries), the argument fails to hold.

\paragraph{Our construction by random partition.} 
Instead of using random projections on the space spanned by the discovered random orthogonal vectors and queried vectors, we adopt a specific random projection parameterized by a random partition over all coordinates.
Specifically, we consider random partitions $\mathcal{P} = (P_1,\ldots,P_{r+1})$ of $[d]$, and define the sum of each part $X^i(\boldsymbol{x}) = \sum_{j\in P_i} \boldsymbol{x}_j$ for query $\boldsymbol{x}$ (where we use $\boldsymbol{a}_i$ to denote the $i$th entry of $\boldsymbol{a}$). 
%
With such random partitions, we construct a similar chaining-like function $f_\mathcal{P}$ with the properties: (a) $\left\|\nabla f_\mathcal{P}(\boldsymbol{x})\right\|$ will not vanish unless $|X^i(\boldsymbol{x})-X^{i+1}(\boldsymbol{x})|>1$ for all $i\in [r]$~(Lemma~\ref{lmm:large_grad}), and (b) the gradient will be independent of $P_{i+1},\ldots,P_r$ if $|X^j(\boldsymbol{x})-X^{j-1}(\boldsymbol{x})|<1/2$, for $j\geq i$~(Lemma~\ref{lmm:character}).
Finally, the key property (c) is that without identifying the elements in $P_i$ and $P_{i+1}$, the value $|X^i(\boldsymbol{x}) - X^{i+1}(\boldsymbol{x})|<1/2$ with a probability $1-d^{-\omega(1)}$.
With such a high probability, even with polynomial times of ``shots in the dark'', the algorithm cannot process information of the future part of the partition.
Thus, any parallel algorithm can only learn one part of the random partition with high probability.
To prove property (c), we utilize the concentration bound of conditional Bernoulli random variables~(Theorem~\ref{theo:concentration2}).



\paragraph{Comparison with construction for stochastic settings.}
The only existing lower bound for non-convex optimization with a large batch of queries pertains to the stochastic case.
%
\citep{arjevani2023lower} utilized the same chaining-like function and perturbed the noiseless gradient oracle to a stochastic ``zero-chaining'' gradient oracle such that any algorithm can recover information about the future orthogonal vector $\boldsymbol{u}^j$ from the noisy oracle at iteration $t<j$ with an exponentially small probability.
Furthermore, they effectively controlled the variance associated with the stochastic gradient oracle.
However, the noise in the stochastic oracle is crucial for obscuring information, making it inapplicable to our noiseless case.

\subsubsection{Trap the flow within a constant number of iterations in constant-dimensional space}
\label{sec:tech2}
Previous algorithms for finding stationary points in constant-dimensional spaces~\citep{bubeck2019complexity,hollender2023computational} employed a similar flow-trapping framework.  
At a high level, in each iteration, these algorithms identify a search domain containing an $\varepsilon$-stationary point and reduce either the volume or the diameter of the domain by a constant factor.  
By the Lipschitz property of the gradient, if the diameter of the domain is $O(\varepsilon)$, an $\varepsilon$-stationary point can be located within it.  
Moreover, the constant reduction in diameter ensures that $\Theta(\log(1/\varepsilon))$ iterations are sufficient to achieve the desired accuracy.
Our algorithm~(Algorithm~\ref{alg:main}) also proceeds with a similar identification-compression framework.
However, to reduce the volume to $O(\varepsilon)$ in $\Theta(1)$ iterations, we design a new algorithm with more ``traps'' in each iteration to identify much smaller domains, which is one of our main technical contributions.
\begin{figure}[t]
    \centering
    \subfigure[]{
        \includegraphics[width=0.23\textwidth]{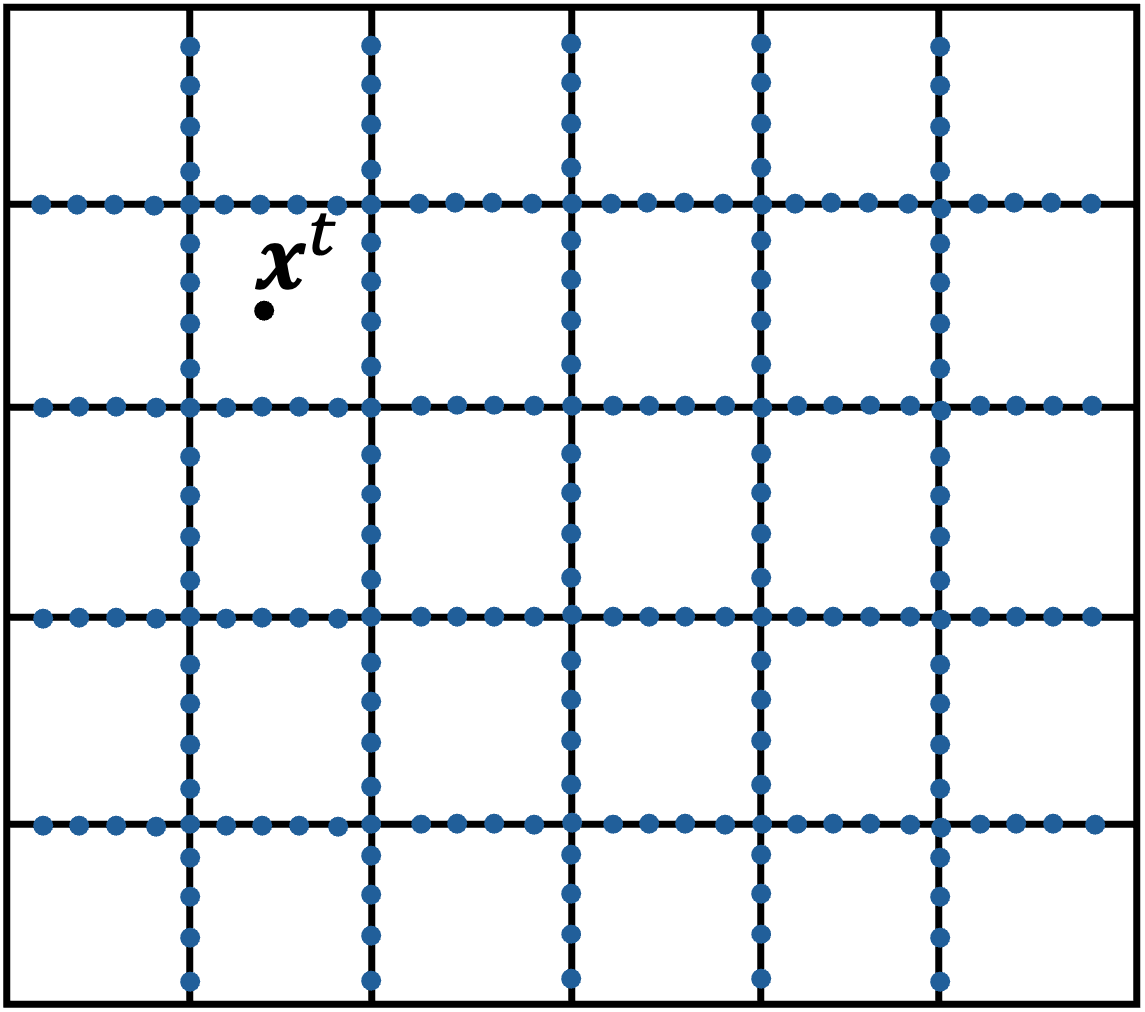}
    }
    \subfigure[]{
        \includegraphics[width=0.23\textwidth]{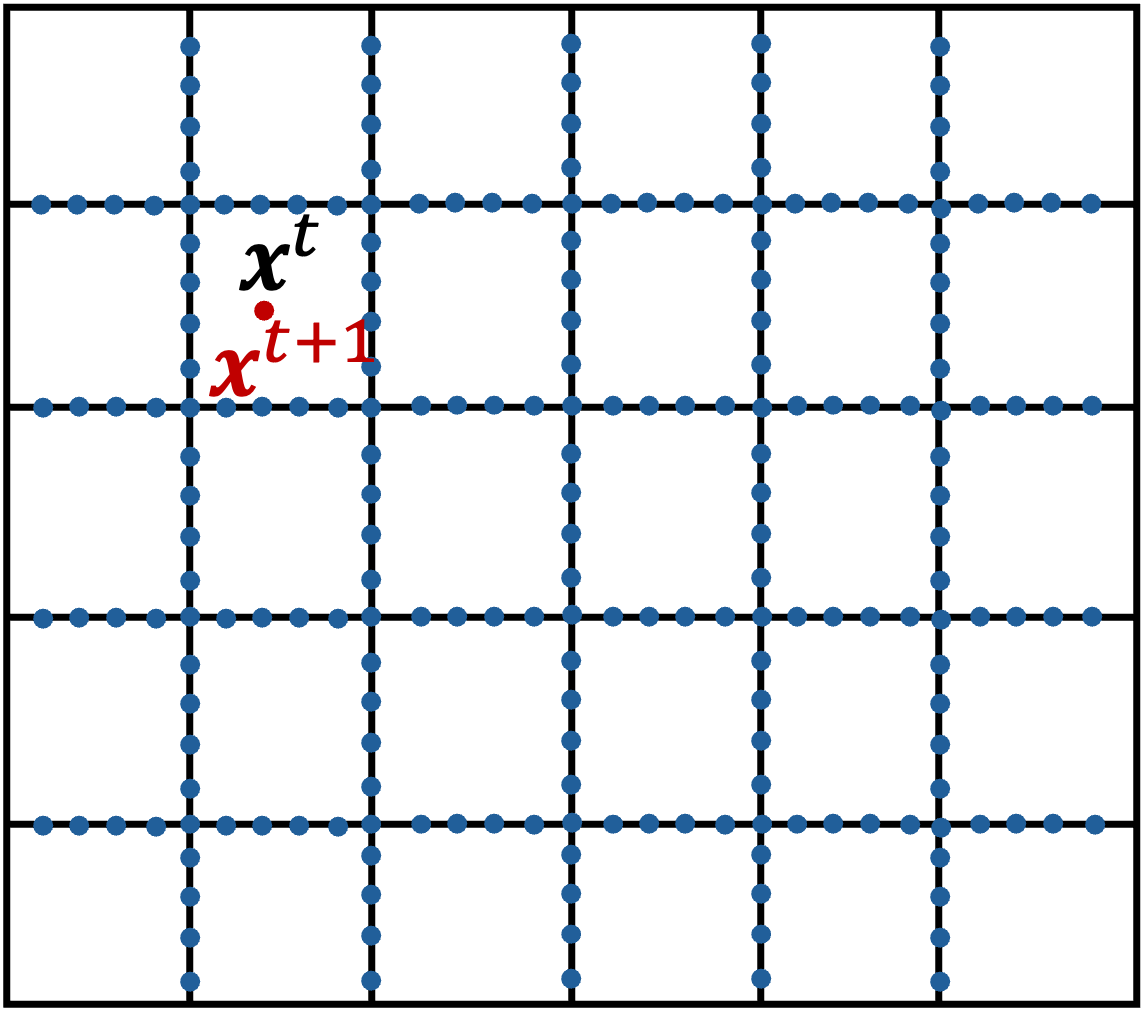}
    }
    \subfigure[]{
        \includegraphics[width=0.23\textwidth]{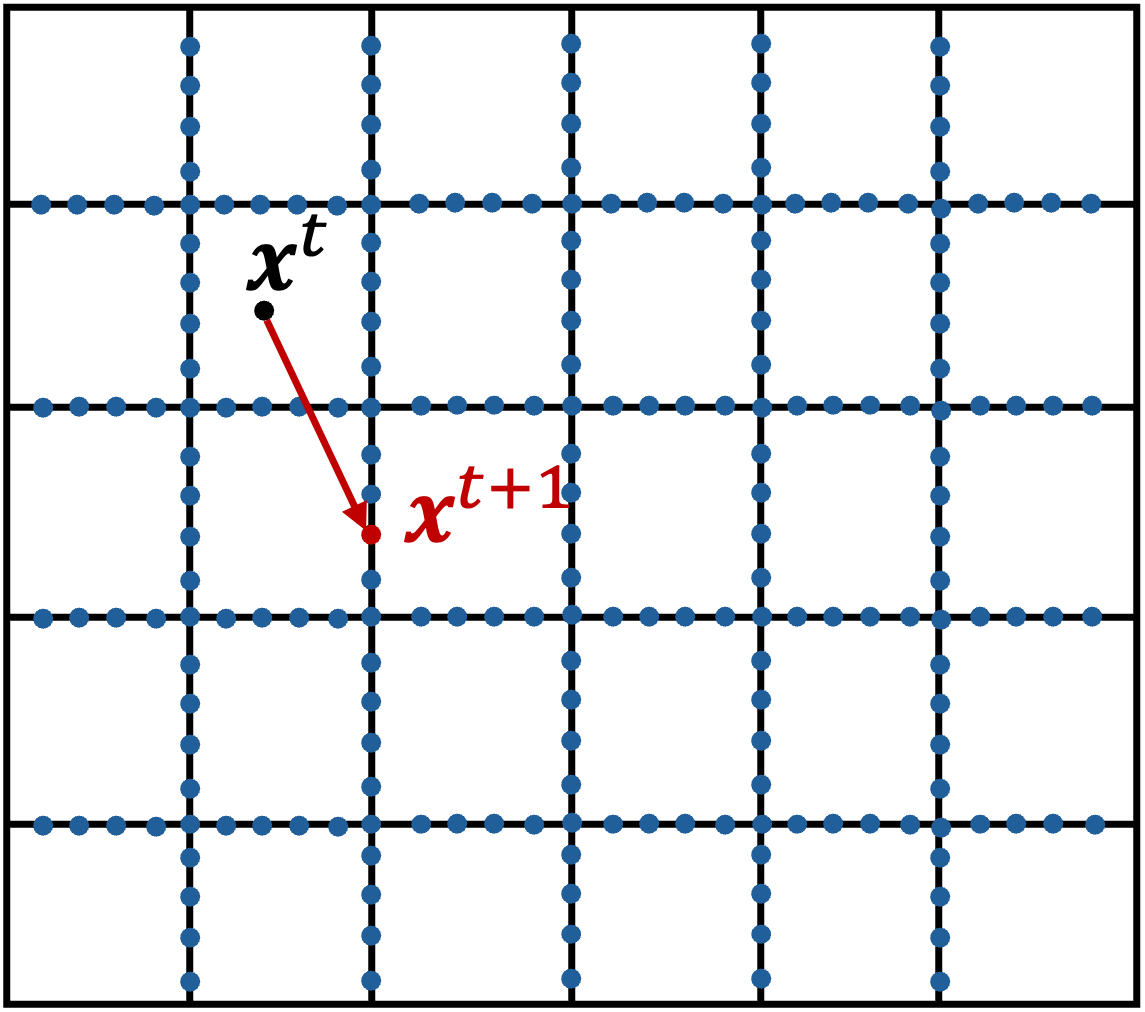}
    }
    \caption{Illustration of \textbf{iterate update} on 2-dimensional space. We plot trap barriers as black grids and queried points on the trap barrier as blue dots. (\emph{a}) Current iterate location. (\emph{b}) If all the queried points are $\varepsilon_t$-unreachable from $\boldsymbol x_t$, then $\boldsymbol x^{t+1} = \boldsymbol x^t$. (\emph{c}) If some queried points are $\varepsilon_t$-reachable, we select $\boldsymbol{x}^{t+1}$ as the point with smallest function value.}
        \label{fig:1}
    \vspace{-0.2cm}
\end{figure}
\paragraph{Boundary unreachability implies stationary point inside.}
%
We say a point $\boldsymbol{y}$ is $\varepsilon$-unreachable from a point $\boldsymbol{x}$ if 
$f(\boldsymbol{y})> f(\boldsymbol{x}) - \varepsilon\left\|\boldsymbol{x}-\boldsymbol{y}\right\|.$
By the mean value theorem, if $\boldsymbol{y}$ is $\varepsilon$-unreachable from $\boldsymbol{x}$, then the gradient flow starting from $\boldsymbol{x}$ and following the steepest descent direction cannot reach $\boldsymbol{y}$, unless it encounters an $\varepsilon$-stationary point along the path.
Conversely, if all boundary points of a domain is $\varepsilon$-unreachable, this domain must contain an $\varepsilon$-stationary point~(Lemma~\ref{lmm:trap}).
With such  a property, if we can compress the region such that the diameter is $O(\varepsilon)$, we can find an $\varepsilon$-stationary point.

\paragraph{Our method: Gradient Flow Grid Trapping~(GFGT).}
The key technical challenge  is to ensure the boundary  remain $\varepsilon$-unreachable while compressing the search domain within a constant number of iterations. 
Inspired by the gradient flow trapping algorithm \citep{bubeck2020trap,hollender2023computational}, we try to trap the flow but with more barriers in each iteration, as shown in Algorithm~\ref{alg:main}.
At a high-level, at iteration $t$, the algorithm first sets $\ell_t$ equally-spaced barriers on every $d$ dimensions and queries a $\delta_t$-net on every barrier. If no queried point is $\varepsilon_t$-reachable from current iterate $\boldsymbol{x}^t$, we do not update $\boldsymbol{x}^t$ and set $\boldsymbol{x}^{t+1} = \boldsymbol{x}^t$. Otherwise, we set $\boldsymbol{x}^{t+1}$ to the reachable point with smallest objective value among queried points.
We illustrate such an identification process in Figure~\ref{fig:1}.
Subsequently, we compress the search domain centered around $\boldsymbol{x}^{t+1}$, while ensuring the boundary remains unreachable from the  $\boldsymbol{x}^{t+1}$. This compression process is visualized in Figure~\ref{fig:2}. Such a construction allows the diameter of the hyperrectangle is reduced by at least $3/\ell_t$. 
Thus, it is sufficient to set $\ell_0\times\cdots\times\ell_k = O(1/\varepsilon)$ to ensure that the diameter is bounded by $O(\varepsilon)$ with $k = \Theta(1)$ number of iterations.

\begin{figure}[t]
    \centering
    \subfigure[]{
        \includegraphics[width=0.228\textwidth]{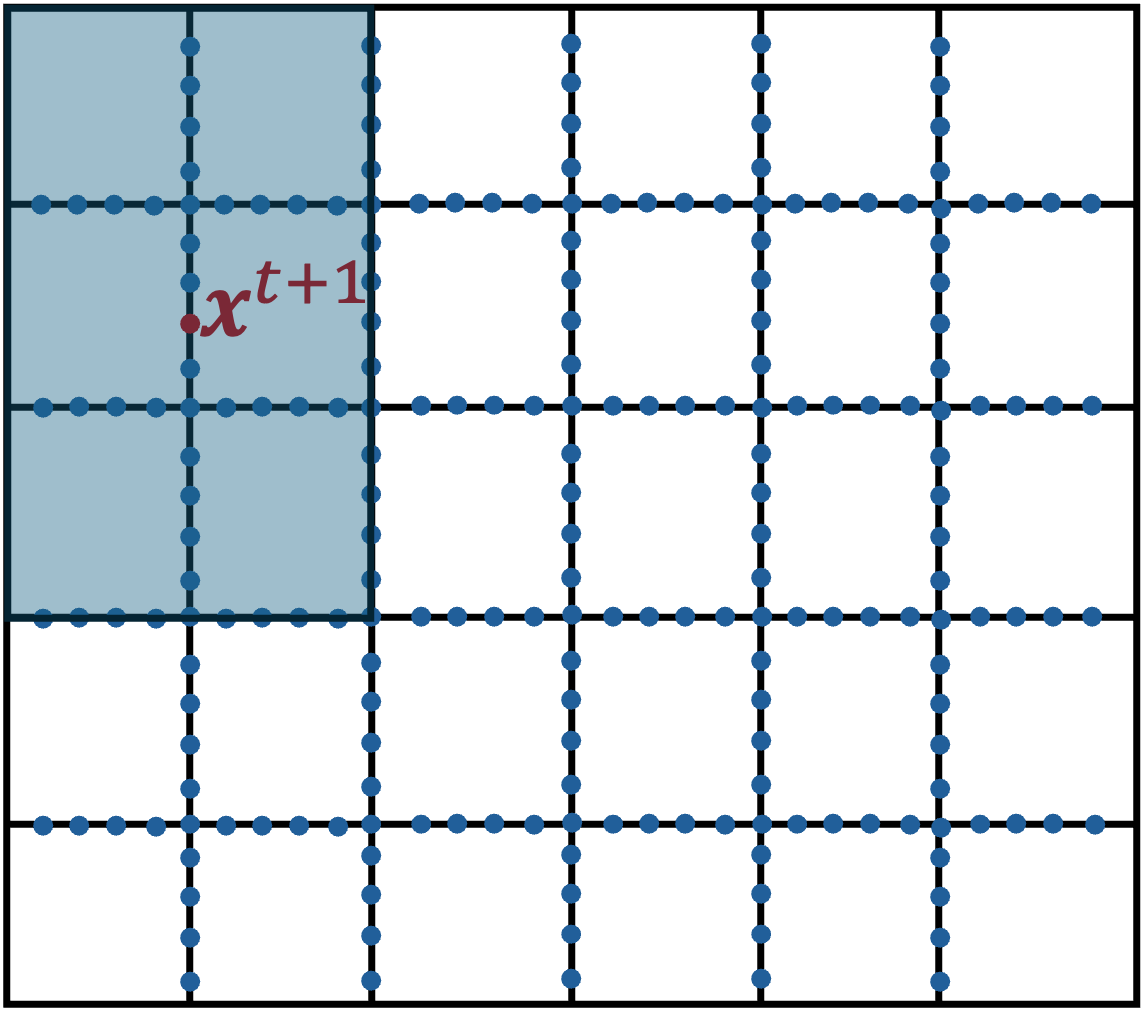}
    }
    \subfigure[]{
        \includegraphics[width=0.228\textwidth]{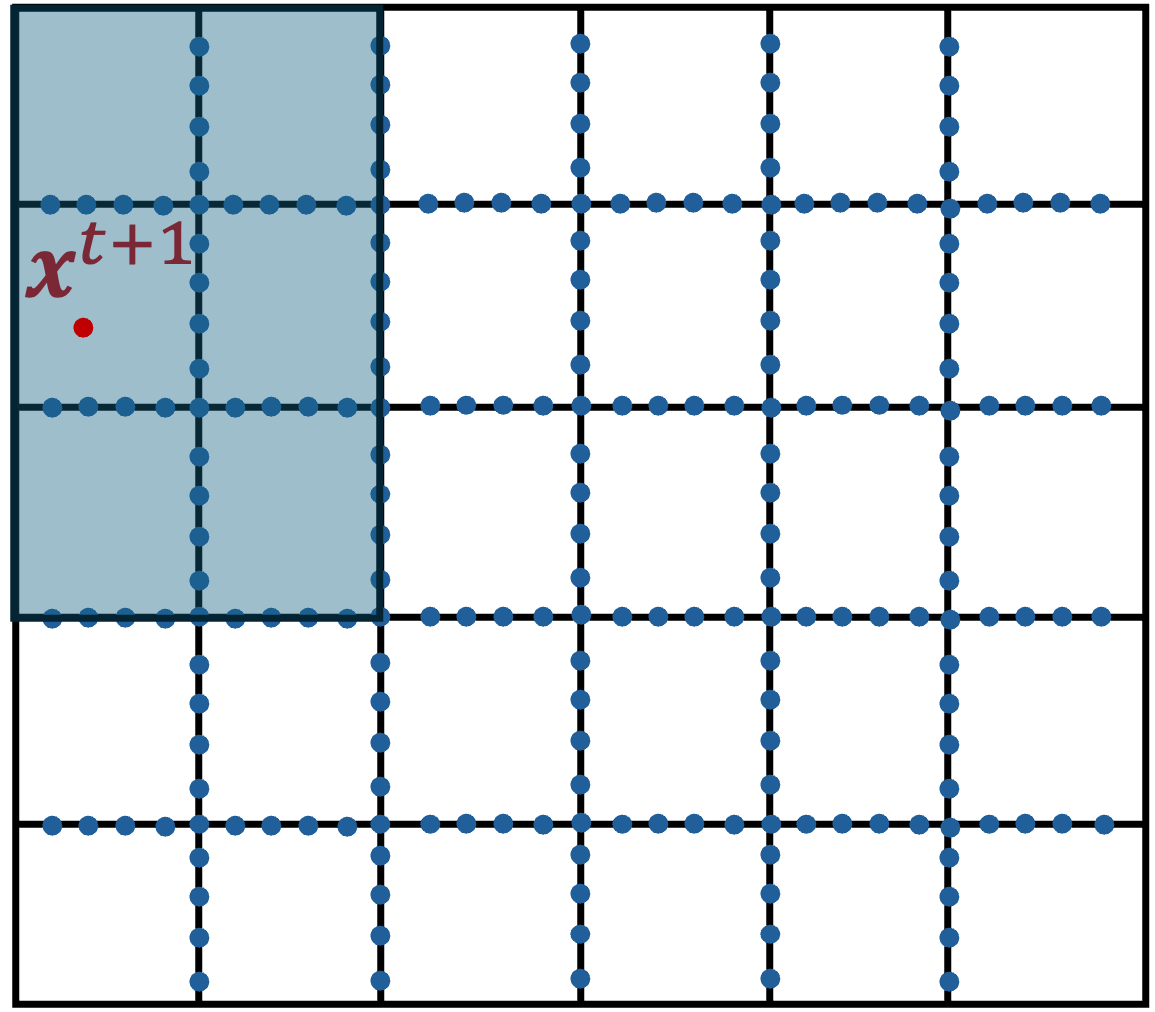}
    }
    \subfigure[]{
        \includegraphics[width=0.228\textwidth]{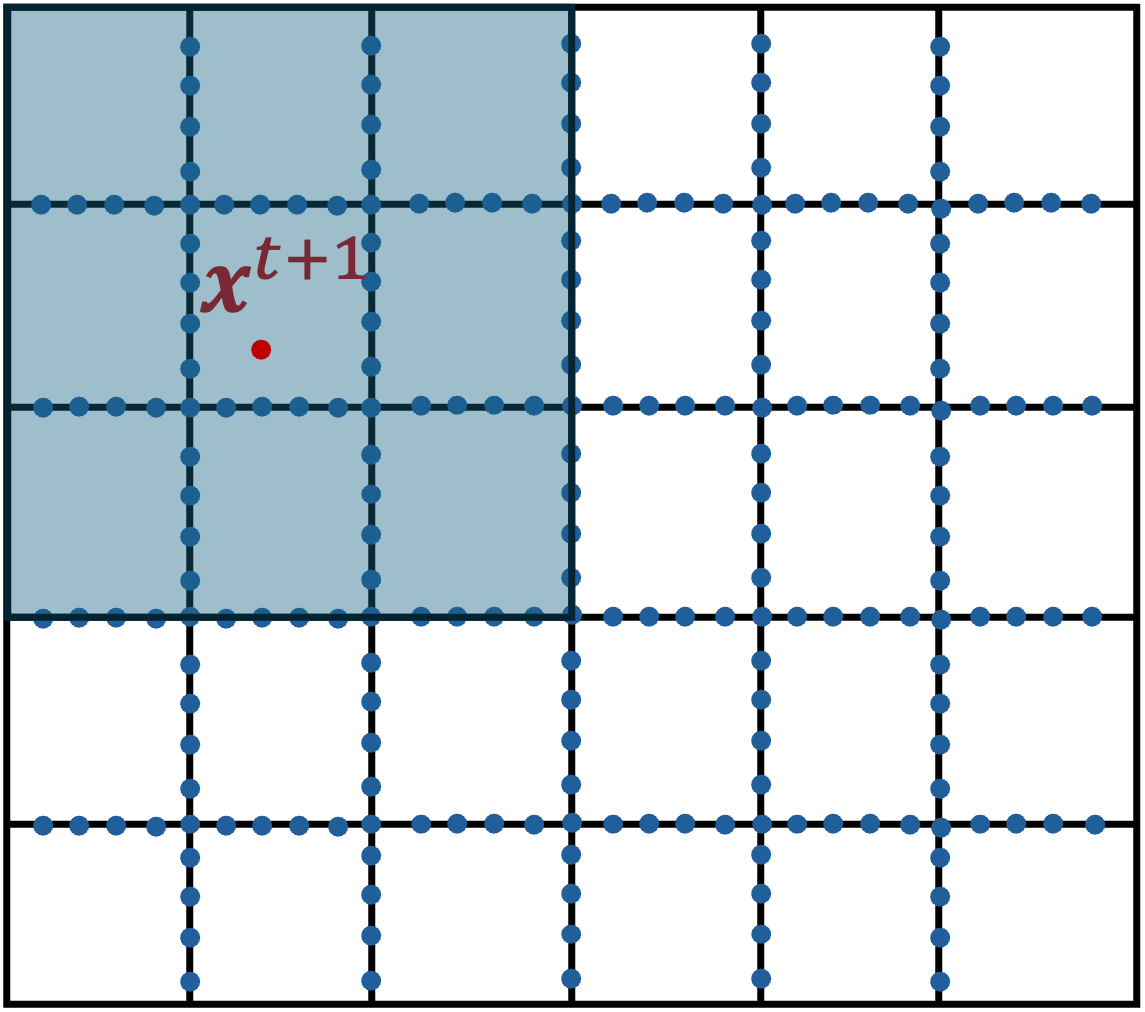}
    }
    \subfigure[]{
        \includegraphics[width=0.228\textwidth]{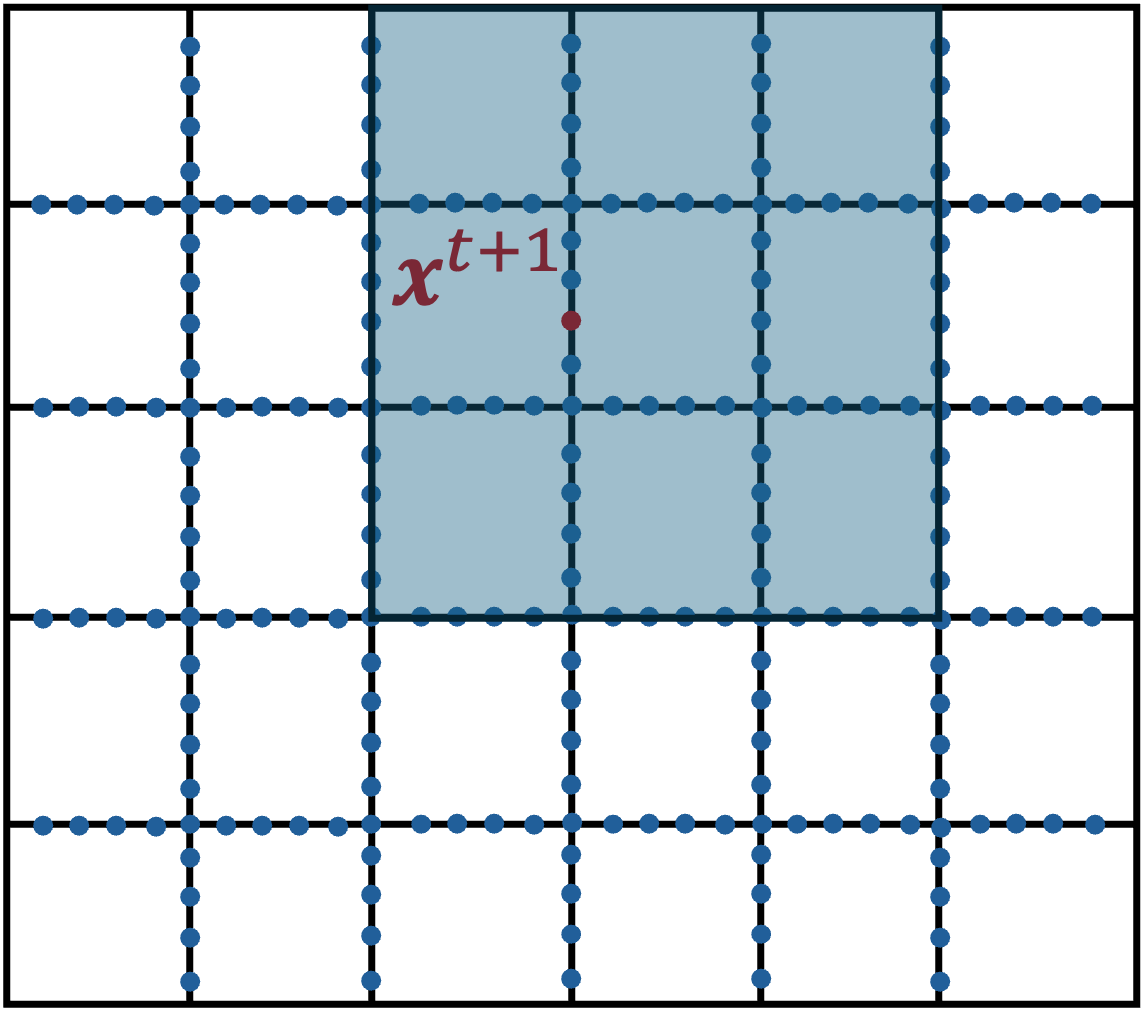}
    }
    \caption{Illustration of \textbf{domain compression} on 2-dimensional space depending on the location of $\boldsymbol{x}^{t+1}$. The shaded areas represent the compressed domain for the next iteration. (\emph{a,b}) If $\boldsymbol{x}^{t+1}$ is close to a boundary (in this case, the left boundary), the domain is extended only to the other directions (top, bottom, right). (\emph{c,d}) If $\boldsymbol{x}^{t+1}$ is not close to any boundary, the domain is extended for to directions, depending on whether $\boldsymbol{x}^{t+1}$ is on the trap barrier.}
        \label{fig:2}
    \vspace{-0.2cm}
\end{figure}

\paragraph{Boundary unreachability during our algorithm.} 
Let $R_{t+1}$ denote the new hyperrectangle at the start of $(t+1)$-th iteration.
We make the following two key observations regarding the iterate $\boldsymbol{x}^{t+1}$ and the hyperrectangle $R_{t+1}$. {(a)} any point on $\partial R_{t+1}$ either belongs to a barrier controlled by a $\delta_t$-net where every point has been queried or lies on the previous boundary $\partial R_t$ and {(b)} the solution $\boldsymbol{x}^{t+1}$ is always sufficiently far  from the boundary $\partial R_{t+1}$.
To leverage the finite $\delta_t$-net in controlling an infinite number of points on $\partial R_{t+1}$, we slightly increase the accuracy $\varepsilon_{t+1} = \varepsilon_t + O(\varepsilon)$. We analyze the unreachability of boundary points under the following scenarios.
\begin{itemize}[leftmargin=0.2in]
    \vspace{-0.05in}
    \item If the point lies in the intersection $\partial R_t\cap \partial R_{t+1}$, we establish $\varepsilon_{t+1}$-unreachability by the triangle inequality, since any point on $\partial R_t$ is already $\varepsilon_t$-unreachable. 
    \vspace{-0.1in}
    \item If the point lies on the new boundary, it can be controlled by a $\delta_t$-net where every point has been queried.
    \vspace{-0.1in}
    \begin{enumerate}[leftmargin=0.2in]
        \item If the iterate is not updated, we show $\varepsilon_{t+1}$-unreachability by Lemma~\ref{lem:increase-eps} and Observation (b). 
    \vspace{-0.08in}
    \item If the iterate is updated to the minimal-value one among reachable all queried points, then the other queried point must be $\varepsilon_t$-unreachable from the minimal-value solution. Then we apply the same argument to show the $\varepsilon_{t+1}$-unreachability.
    \end{enumerate}
\end{itemize}

\paragraph{Comparison with constant-iteration algorithms for local minimum search on grid graphs.} 
Since the problem of local minimum search on grid graphs can be reduced to finding a stationary point~\citep{bubeck2020trap, vavasis1993black}, constant-iteration algorithms for local minimum search on grid graphs share similar structural features.
In particular, these algorithms query \emph{all} points on the boundary of a subregion to progressively trap the local minimum within a significantly shrinking region. They then proceed to query {all} points within the final subregion.
However, to trap a stationary point, one must ensure that all boundary points are {unreachable}.
Since directly querying all points on the boundary of a sub-region would require infinitely many queries, we instead adopt the approach of \citep{hollender2023computational}, which ensures boundary unreachable by placing the next center point sufficiently far from the boundary. To implement this, we query a $\delta$-net with spacing smaller than $\varepsilon$, thereby ensuring the stationary point remains unreachable via any finite number of queries.

\paragraph{Constant-dimensional lower bound.}

To analyze the constant-dimensional case, we reduce the problem to local search on grid graphs and leverage the hardness of random staircase construction. Specifically, we show that the problem of finding stationary points can be reduced to locating the local minimum of a monotone path function on a grid graph, following \citep{bubeck2020trap,vavasis1993black}. A smooth function is constructed to preserve the path's monotonicity, ensuring the stationary point lies near the local minima. Notably, such a reduction is round-preserving. The adaptive complexity is analyzed using random staircase construction \citep{branzei2022query}, where the staircase grows randomly at each step, restricting deterministic algorithms from accessing points beyond the current round with constant probability. 
We demonstrate that these hard instances hold for any algorithm running within $\Theta(\log(1/\varepsilon))$ rounds.









\section{Preliminaries}
\label{sec:pre}

Given a smooth objective function $f:\mathbb{R}^d\to \mathbb{R}$,  the goal of non-convex optimization is to find a (possibly random) stationary point $\boldsymbol{x}\in \mathbb{R}^d$ such that $\mathbb{E}\left[\left\|\nabla f (\boldsymbol{x})\right\|\right]\leq \varepsilon$. We further assume the initial point $\boldsymbol{x}^0$ satisfies $f(\boldsymbol{x}^0) - \inf_{\boldsymbol{x}} f(\boldsymbol{x})\leq \Delta$ for some $\Delta\geq 0$. In the constant dimensional setting, we also consider the cube-constrained minimization where $f$ is restricted to $[0,1]^d$. In this case, we focus on finding an $\varepsilon$-KKT point, 
\emph{i.e.}, the norm of the projected gradient instead of the gradient being bounded by $\varepsilon$ \citep{hollender2023computational}.

\paragraph{Function class.} A function $f:\R^d \to \R$ has $L_p$-Lipschitz $p$-th order derivatives if it is $p$ times continuously differentiable, and 
for every $\boldsymbol{x} \in \R^d$ and $\boldsymbol{v} \in \R^d$, $\left\|{\boldsymbol{v}}\right\| = 1$, the directional projection
$t \mapsto f_{\boldsymbol{x},\boldsymbol{v}}(t) := f(\boldsymbol{x} + t \cdot \boldsymbol{v})$ of $f$
satisfies
$
  \left| f_{\boldsymbol{x},\boldsymbol{v}}^{(p)}(t) - f_{\boldsymbol{x},\boldsymbol{v}}^{(p)}(t') \right|
  \le L_p \left|t - t'\right|
  ~~ \mbox{for~} t, t' \in \R,
$
where $f_{\boldsymbol{x},\boldsymbol{v}}^{(p)}(\cdot)$ is the $p$th
derivative of $t \mapsto f_{\boldsymbol{x},\boldsymbol{v}}(t)$. We occasionally refer to a function with Lipschitz continuous  $p$-th order derivatives as $p$-th order smooth and denote the class of such functions by $\mathcal{F}_p(L_p)$.

\paragraph{Oracle.}
In this work, we investigate the model where the algorithm queries points to the oracle $\mathcal{O}$.
Given the potential function $f$, and a query $\boldsymbol{x} \in \mathbb{R}^d$, the $0$-th order oracle answers the function value $f(\boldsymbol{x})$ and the $p$-th order oracle answers both $f(\boldsymbol{x})$ and its $i$-th order derivative $\nabla^i f(\boldsymbol{x})$, for any $i\in [p]$.
%
%
In the following, we focus solely on the zeroth-order oracle, as higher-order oracles can be constructed using $\mathsf{poly}(d)$ queries to the zeroth-order oracle.


\paragraph{Adaptive algorithm class.} 
The class of \emph{adaptive} algorithms is formally defined as follows~\citep{diakonikolas2019lower,zhou2024adaptive}.
For any dimension $d$, an adaptive algorithm $\mathsf{A}$ takes $f:\mathbb{R}^d\to \mathbb{R}$ and a (possibly random) initial point $\boldsymbol{x}^0$ and iteration number $r$ as input
and returns an output $\boldsymbol{x}^{r+1}$, which is denoted as $\mathsf{A}[f,\boldsymbol{x}^0,r] = \boldsymbol{x}^{r+1}$.
At iteration $i\in [r]:=\{1,\ldots,r\}$, $\mathsf{A}$ performs a batch of queries
\[Q^i=\{\boldsymbol{x}^{i,1},\ldots,\boldsymbol{x}^{i,k_i}\}, ~~\mbox{ with } \boldsymbol{x}^{i,j}\in \mathbb{R}^d,~~ j\in [k_i], ~~k_i = \mathsf{poly}\big(d,1/\varepsilon\big) ,\]
%
%
Given queries set {$Q^i$}, the oracle returns a batch of answers: {$\mathcal{O}(Q^i)  = \{\mathcal{O}(\boldsymbol{x}^{i,1}),\ldots,\mathcal{O}(\boldsymbol{x}^{i,{k_i}})\}.$}

An adaptive algorithm $\mathsf{A}$ is \emph{deterministic} if in every iteration $i\in \{0,\ldots,r\}$, $\mathsf{A}$ operates with the form $Q^{i+1} =\mathsf{A}^i(Q^0,\mathcal{O}(Q^0),\ldots,Q^{i},\mathcal{O}(Q^{i})),$
where $\mathsf{A}^i$ is mapping into $\mathbb{R}^{dk_{i+1}}$ with $Q^{r+1} = \boldsymbol{x}^{r+1}$ as output and $Q^0 =\boldsymbol{x}^0$ as an initial point.
We denote the class of adaptive deterministic algorithms by $\mathcal{A}_{\text{det}}$.
An adaptive \emph{randomized} algorithm has the form $Q^{i+1} =\mathsf{A}^i(\xi_i,Q^0,\mathcal{O}(Q^0),\ldots,Q^{i},\mathcal{O}(Q^{i})), $
{with} access to a {uniform random variable} on $[0, 1]$ (\emph{i.e.}, infinitely many random bits), where $\mathsf{A}^i$ is mapping into $\mathbb{R}^{dk_{i+1}}$.
We denote the class of adaptive randomized algorithms by $\mathcal{A}_{\text{rand}}$.


\paragraph{Notion of complexity}
Given $\varepsilon > 0$, $f\in \mathcal{F}$, and some algorithm $\mathsf{A}$, 
define the running iteration 
$\mathsf{T}(\mathsf{A}, f,\boldsymbol{x}^0,\varepsilon)$ as the minimum number of rounds such that algorithm $\mathsf{A}$ outputs a solution $\boldsymbol{x}$ that satisfies $\mathbb{E}\left[\left\|\nabla f (\boldsymbol{x})\right\|\right]\leq \varepsilon$, \emph{i.e.}, $\mathsf{T}(\mathsf{A}, f,\boldsymbol{x}^0,\varepsilon) = \inf\left\{t: \mathbb{E}\left[\left\|\nabla f \left(\mathsf{A}[f,\boldsymbol{x}^0,t]\right)\right\|\right]\leq \varepsilon \right\}$.
We define the \emph{worst case} complexity as
\[\mathsf{Comp}_{\mathsf{WC}}(\mathcal{F},\varepsilon,\boldsymbol{x}^0) :=  \inf\nolimits_{\mathsf{A}\in \mathcal{A}_{\text{det}}} \sup\nolimits_{f\in \mathcal{F}}\mathsf{T}(\mathsf{A}, f,\boldsymbol{x}^0,\varepsilon).\]
For some randomized algorithm $\mathsf{A}\in \mathcal{A}_{\mathrm{rand}}$, we define the \emph{randomized} complexity as
\[\mathsf{Comp}_{\mathsf{R}}(\mathcal{F},\varepsilon,\boldsymbol{x}^0) :=  \inf\nolimits_{\mathsf{A}\in \mathcal{A}_{\text{rand}}} \sup\nolimits_{f\in \mathcal{F}}\mathsf{T}(\mathsf{A}, f,\boldsymbol{x}^0,\varepsilon).\]
By definition, we have $\mathsf{Comp}_{\mathsf{WC}}(\mathcal{F},\varepsilon,\boldsymbol{x}^0) \geq  \mathsf{Comp}_{\mathsf{R}}(\mathcal{F},\varepsilon,\boldsymbol{x}^0).$
 In the rest of this paper, we only consider the randomized complexity and we lower-bound it by considering the \emph{distributional} complexity:
\[\mathsf{Comp}_{\mathsf{D}}(\mathcal{F},\varepsilon,\boldsymbol{x}^0) :=  \sup\nolimits_{F\in \Delta(\mathcal{F})}\inf\nolimits_{\mathsf{A}\in \mathcal{A}_{\text{rand}}} \mathop{\mathbb{E}}\nolimits_{f\sim F}\mathsf{T}(\mathsf{A}, f,\boldsymbol{x}^0,\varepsilon),\] 
where $\Delta(\mathcal{F})$ is the set of probability distributions over the class of functions $\mathcal{F}$.


\section{High dimensional lower bounds for $p$-th order methods}
\label{sec:high}

In this section, we prove the adaptive lower bound for $d = \widetilde{\Omega}\big(\varepsilon^{-\frac{2+2p}{p}}\big)$~(Theorem \ref{the:main}), where the number of queries per iteration is dominated by $\mathsf{poly}(d)$.
%
We first present our hard functions family and analyze their properties in Section \ref{sec:hard1}. 
We characterize the output for any algorithm and give proof of Theorem \ref{the:main} in Section~\ref{sec:char} .
%

\begin{theorem}
\label{the:main}
There exist numerical constants $0 < c_0, c_1,c_2 < \infty$
  such that  for all $p \ge 1, p \in \mathbb{N}$,
  and let $\Delta$, $L_p$, $\varepsilon>0$ and $d$ with $\frac{d}{\log^2 d}\geq c_2\big(\frac{L_p}{l_p}\big)^{2/p} \varepsilon^{-\frac{2+2p}{p}}$, we have
  \begin{equation*}
    \mathsf{Comp}_{\mathsf{R}}(\mathcal{F}_p(L_p),\varepsilon,\Delta)\geq c_0\cdot \Delta \bigg(\frac{L_p}{l_p}\bigg)^{1/p} \varepsilon^{-\frac{1+p}{p}},
  \end{equation*}
  where $l_p \le e^{c_1 p \log p + c_1}$.
\end{theorem}
Compared to the query complexity presented in Theorem 2 of \citep{carmon2020lower}, our adaptive complexity achieves the same query complexity but with only a logarithmic dependence on the dimension.
This logarithmic scaling arises from our use of random partitions to hide information, where the logarithmic size of each part of partition is essential for ensuring concentration in high-dimensional space.

\subsection{Hard functions and its properties}
\label{sec:hard1}
Let $d_0\in \mathbb{N}_{+}$ with $ d_0 \geq \log^2 d$, and $r = d/d_0-2$.
Consider partition with fixed-size components as $P_{1}\cup P_{2} \cup \cdots\cup P_{r+2}= [d]$, where $ |P_{1}| =d_0$, $|P_{1}| = |P_{i}|$ for all $i\in [r] $.
The partition is uniformly random among all partitions with such fixed-size parts, and we denote it as $\mathcal{P}$.
For any $\boldsymbol{x}\in \mathbb{R}^d$ let  {$X^{i} (\boldsymbol{x})= (\sum_{s\in P_{i}}\boldsymbol{x}_s)/\sqrt{d_0}$}, for all $i\in [r+1]$, and we denote {$X(\boldsymbol{x}) = (X^1(\boldsymbol{x}),\ldots,X^{r+1}(\boldsymbol{x}))$}.
We define the (unscaled) hard function
$f_\mathcal{P}:\mathbb{R}^d \to \mathbb{R}$ as,
\[f_\mathcal{P} (\boldsymbol{x}) = g_\mathcal{P}(\rho(\boldsymbol{x})) + \frac{1}{5}\left\|\boldsymbol{x}\right\|^2\,,\]
where $\rho(\boldsymbol{x}) = \frac{\boldsymbol{x}}{\sqrt{1+\norm{\boldsymbol{x}}^2/R^2}} 
  ~\mbox{and}~R=230\sqrt{r+1}\,$ and 
  \begin{align*}
g_\mathcal{P}(\boldsymbol{x})= -\Psi(1)\Phi(X^1)- \sum\limits_{i=1}^{r}(-1)^i\left[\Psi(X^{i-1} - X^i)\Phi(  X^{i+1} - X^{i}) - \Psi(X^i-X^{i-1})\Phi( X^{i}-X^{i+1}) \right]\,,      
  \end{align*}
where $X^0 \equiv 0$ and the component functions are defined as \citep{carmon2020lower},
\begin{equation*}
  \Psi(x)
  := \begin{cases}
  0 & x \le 1/2\\
   \exp\left(1-\frac{1}{\left(2x-1\right)^{2}}\right)  & x>1/2
  \end{cases}
  ~~\mbox{and}~~
  \Phi(x)
  = \sqrt{e} \int_{-\infty}^{x} e^{-\frac{1}{2} t^{2}}\mathrm{d}t\,.
\end{equation*}
We enumerate all the relevant properties of the component functions $\Psi$ and $\Phi$ in the following.
\begin{lemma}[\textbf{\cite[Lemma 1]{carmon2020lower}}]
	\label{lem:fullder-props}
  The functions $\Psi$ and $\Phi$ satisfy the following.
  \begin{enumerate}
  \item \label{item:fullder-psiphi-props-zero}
    For all $x \le 1/2$ and
    all $k \in [N]$, $\Psi^{(k)}(x) = 0$.
  \item \label{item:fullder-psiphi-props-product}
    For all $x \ge 1$ and $|y| < 1$,
    $\Psi(x)\Phi'(y) > 1$.
  \item\label{item:fullder-psiphi-props-infinite}
    Both $\Psi$ and $\Phi$
    are infinitely differentiable,
    and for all $k \in [N]$ we have
    \begin{equation*}
      \sup\nolimits_x |\Psi^{(k)}(x)|
      \le \exp\bigg(\frac{5 k}{2}\log(4 k)\bigg)
      ~~\mbox{and}~~
      \sup\nolimits_x |\Phi^{(k)}(x)|
      \le \exp\bigg(\frac{3k}{2} \log \frac{3k}{2} \bigg).
    \end{equation*}
  \item \label{item:fullder-psiphi-props-bounded} The functions and derivatives
    $\Psi, \Psi', \Phi$ and $\Phi'$ are
    non-negative and bounded, with
    \begin{equation*}
      0 \le \Psi < e,
      ~~ 0 \le \Psi' \le \sqrt{54/e},
      ~~ 0 < \Phi < \sqrt{2\pi e},
      ~~ \mbox{and} ~~
      0 < \Phi' \le \sqrt{e}.
    \end{equation*}
\end{enumerate}
\end{lemma}
By careful calculation and Lemma~\ref{lem:fullder-props}, we can obtain the smoothness and boundedness of $f_\mathcal{P}$ and $g_\mathcal{P}$. 
We defer the proof in Appendix~\ref{app:g} and Appendix~\ref{app:f}.
\begin{lemma}[\textbf{Smoothness and boundness of $g_\mathcal{P}$}]
\label{lmm:g}
The function $g_\mathcal{P}$ satisfies the following.
\begin{enumerate}
\item We have $g_\mathcal{P}(\boldsymbol{0})-\inf_{\boldsymbol{x}}g_\mathcal{P}(\boldsymbol{x})< 12 r $;
\item For all $\boldsymbol{x}\in \mathbb{R}^d$, $\left\|\nabla g_\mathcal{P}(\boldsymbol{x}) \right\|\leq 46\sqrt{r+1}$.
\end{enumerate}
\end{lemma}

\begin{lemma}[\textbf{Smoothness and boundness of $f_\mathcal{P}$}]
\label{lmm:f_p}
The function $f_\mathcal{P}$ satisfies the following,
\begin{enumerate}
\item We have $f_\mathcal{P}(\boldsymbol{0}) - \inf_{\boldsymbol{x}}f_\mathcal{P}(\boldsymbol{x})<12r$.
\item For every $p \geq 1$, the $p$-th order derivatives of $f_\mathcal{P}$ are $l_p$-Lipschitz continuous, where $l_p \leq
\exp(cp \log p + c)$ for a numerical constant $c < \infty$.
\end{enumerate}
\end{lemma}

\subsection{Characterization of output}
\label{sec:char}
We now turn to analyze the output when interacting with $f_\mathcal{P}$ in the following lemma. The proof can be found at Appendix~\ref{app:char}.
\begin{lemma}[\textbf{Characterization of output}]
\label{lmm:character}
If there exists a constant $\alpha= \omega(1)$ such that $8R\sqrt{\frac{\alpha\log d}{d_0}}\leq \frac{1}{2}$, then 
for any randomized algorithm $\mathsf{A}$, any $\tau \leq r$,  and any initial point $\boldsymbol{x}^0$, 
$X(\mathsf{A}[f_\mathcal{P},\boldsymbol{x}^0,\tau]) = (X^1(\rho(\boldsymbol{x})),\ldots,X^{r+1}(\rho(\boldsymbol{x})))$ takes a form as 
\[(x_1,\ldots,x_{2\tau},x_{2\tau+1},x_{2\tau+1},\ldots,x_{2\tau+1}),\]
up to addictive error $1/4$ {for every coordinate} with probability $1-d^{-\omega(1)}$ over $\mathcal{P}$. 
\end{lemma}

\begin{proof}(Proof sketch of Lemma~\ref{lmm:character}).
 We fix $\tau$ and prove the following by induction for $l\in [\tau]$: With high probability, the computation
path of the (deterministic) algorithm $\mathsf{A}$ and the queries it issues in the $l$-th round are determined by $P_{1},\ldots,P_{2l-2}$.   
We first consider deterministic algorithm and uniformly random partition.

To prove the inductive claim, we let $\mathcal{E}_l$ as event that any answer of query issued in iteration $l$, the answer only depends on $P_1,\ldots,P_{2l}$, \emph{i.e.}, $\forall \boldsymbol{x} \in Q^l$, $f_\mathcal{P}(\boldsymbol{x}) = g_\mathcal{P}^l(\rho(\boldsymbol{x})) + \frac{1}{5} \| \boldsymbol{x}\|^2$. Combining the assumption that the queries in round $l$ depend only on $P_{1},\ldots,P_{2l-2}$, if $\mathcal{E}_l$ occurs, the computation path is determined by $P_{1},\ldots,P_{2l}$. Thus if all of $\mathcal{E}_1, \ldots , \mathcal{E}_{l}$ occur, the computation path in round $l$ is determined by $P_{1},\ldots,P_{2l}$.

It is sufficient to show the conditional probability $\mathbb{P}\left[\mathcal{E}_l\mid \mathcal{E}_1,\ldots,\mathcal{E}_{l-1}\right]$ is at least $1-d^{\omega(1)}$ to hide information from polynomial queries.
Further, by the chaining structure of hardness function and $\Psi(x) = 0$ whenever $x<1/2$, it is sufficient to show for any fixed $i\geq 2l$, with probability at least $1-d^{\omega(1)}$,
\[\left|X^{i}(\rho(\boldsymbol{x}))-X^{i+1}(\rho(\boldsymbol{x}))\right|\leq \frac{1}{2}.\]
This can be guaranteed by uniformly random partition and the {concentration of linear functions over the Boolean slice} (Theorem \ref{theo:concentration2}). Finally, we note that by allowing the algorithm to use random bits, the results are
a convex combination of the bounds above, so the same high-probability bounds are satisfied.
\end{proof}

%



Given such characterization, the remaining goal is to show that without final part being discovered, the gradient norm of output will always be large.
We summarize this in the following lemma and the proof can be found in Appendix~\ref{app:large_grad}. 
\begin{lemma}[\textbf{Small weighted partition implies large gradient norm}]
\label{lmm:large_grad}
For any $\boldsymbol{x}\in \mathbb{R}^d$, if $|X^{r+1}(\rho(\boldsymbol{x}))- X^{r}(\rho(\boldsymbol{x}))|<1$, then $\left\|\nabla f_\mathcal{P}(\boldsymbol{x})\right\|\geq 0.08$.
\end{lemma}

Now we are ready to prove Theorem \ref{the:main} by leveraging Lemma~\ref{lmm:character} and Lemma~\ref{lmm:large_grad}.

\begin{proof}(Proof of Theorem \ref{the:main}).
We consider the scaled hard function $f^0_\mathcal{P}:\mathbb{R}^d \to \mathbb{R}$ as 
\[f^0_\mathcal{P}(\boldsymbol{x}) = \frac{L_p\sigma^{p+1}}{l_p}f_\mathcal{P}(\boldsymbol{x}/\sigma), \]
where scale parameter $\sigma > 0$ are to be
determined, 
and the
quantity $l_{p} \le \exp(c_1 p \log p + c_1)$ for a numerical constant $c_1$
is defined in
Lemma~\ref{lmm:f_p}.

Combining Lemma~\ref{lmm:character} and Lemma~\ref{lmm:large_grad}, we can claim if there exists a constant $\alpha= \omega(1)$ such that $8R\sqrt{\frac{\alpha\log d}{d_0}}\leq \frac{1}{2}$,
then 
\[\left\|\nabla f_\mathcal{P}(\boldsymbol{x}/\sigma)\right\|\geq 0.08,\]
with probability $1-d^{-\omega(1)}$ over $\mathcal{P}$. Taking $\sigma = \left(\frac{l_p\varepsilon}{0.08L_p}\right)^{1/p}$, we have 
\[\mathsf{Comp}_{\mathsf{R}}(\mathcal{F},\varepsilon,\Delta)\geq(1-d^{-\omega(1)})\cdot r/2.  \]
Then, we have 
\[f^0_\mathcal{P}(\boldsymbol{0}) - \inf_{\boldsymbol{x}}f^0_\mathcal{P}(\boldsymbol{x})\leq \frac{L_p\sigma^{p+1}}{l_p} 12 r\leq 1857\frac{l_p^{1/p}}{L_p^{1/p}}\varepsilon^{(p+1)/p}r. \]
Then we take $r = \lfloor{\frac{\Delta}{1857} ({L_p}/{l_p})^{1/p} \varepsilon^{-\frac{1+p}{p}}}\rfloor$.

Finally, we check the requirement of dimension $d$ to satisfy that there exists a constant $\alpha= \omega(1)$ such that $8R\sqrt{\frac{\alpha\log d}{d_0}}\leq \frac{1}{2}$. It is sufficient to let 
$16^2\cdot 230^2 (r+1)^2{{\log^2 d}}\leq d.$
\end{proof}

\section{Constant dimensional cases}
\label{sec:constant}
In this section, we first establish the lower bound on the number of queries required per iteration for algorithms operating within  $k < O(\log(1/\varepsilon))$  iterations, as detailed in Section \ref{sec:lower_con}.
Subsequently, we derive the upper bound on the number of queries required per iteration for algorithms running within a constant number of iterations, as outlined in Section \ref{sec:upper_con}.

\subsection{Upper bound for constant iteration case}
\label{sec:upper_con}
\begin{theorem}
\label{the:main3}
For $\varepsilon>0$, $d,k = \Theta(1)$ with $d\geq 2$, there is a deterministic algorithm finding $\varepsilon$-stationary points, running within $k$-round, with
\begin{itemize}
    \item $C_1(d,k,L,\Delta)\cdot \varepsilon^{-\frac{{d-1}}{\left(\frac{2d}{d+1}\right)^k-1}-{d-1}}$ queries per iteration when $f$ is unconstrained,
    \item $C_2(d,k,L)\cdot  \varepsilon^{-\frac{{d-1}}{2\left(\left(\frac{2d}{d+1}\right)^k-1\right)}-\frac{d-1}{2}}$ queries per iteration when $f$ is constrained on $[0,1]^d$, 
\end{itemize}
where $C_1(d,k,L,\Delta)$ and $C_2(d,k,L)$ are uniformly constants depending on $d,k,L,\Delta$ and $d,k,L$ respectively.
\end{theorem}
The unconstrained case has square number of queries due to the initialization region scale with $\frac{\Delta}{\varepsilon}$ instead of $[0,1]^d$. 
For constrained case,  our upper bound is asymptotically tight compared to the lower bound $\Omega(\varepsilon^{-\frac{{d-1}}{2(d^k-1)}- \frac{d-1}{2}})$ in Theorem~\ref{the:main2}. We note that numerically, $({d-1})/\big({\big(\frac{2d}{d+1}\big)^k-1}\big)<0.05$ when $d=2, k= 9$.
Furthermore, when $d=2$ and $k=1$, the complexity is  $\varepsilon^{-2}$  which matches the complexity of a na\"ive grid search.
These upper bounds are first non-trivial results for $d,k=\Theta(1)$.
The algorithms and their proofs are deferred to Appendix~\ref{app:constant_alg}.

\subsection{Lower bounds for $O(\log(1/\varepsilon))$ iteration case}
\label{sec:lower_con}

\begin{theorem}
\label{the:main2}
For $\varepsilon>0$, $d = \Theta(1)$ with $d\geq 2$ and $k\in \mathbb{N}$ with $k=O(\log (1/\varepsilon))$, any (possible randomized) algorithm running within $k$-round, and issuing $C(d)\cdot \varepsilon^{-\frac{d^k}{d^k-1}\cdot \frac{d-1}{2}}\cdot \frac{1}{k}$ queries per round fails to find $\varepsilon$-stationary points of a smooth function over $[0,1]^d$ with probability at least $7/40$.
\end{theorem}
When $k = \Theta(\log (1/\varepsilon))$, the lower bound is 
\[\Omega\bigg( \varepsilon^{-\left(1+\frac{1}{d^k-1}\right)\cdot \frac{d-1}{2}}\cdot \frac{1}{k}\bigg) = \Omega\left( \varepsilon^{-\frac{d-1}{2}}\cdot \varepsilon^{-\frac{d-1}{2(d^k-1)}}\cdot \frac{1}{k}\right) =  \widetilde{\Omega}\left( \varepsilon^{-\frac{d-1}{2}}\right),\]
the last equality is implied from $\varepsilon^{-1}\leq (\log(1/\varepsilon))^{\frac{2(d^k-1)}{d-1}}$ and $\varepsilon^{-\frac{d-1}{2(d^k-1)}}\cdot \frac{1}{k}\geq \frac{1}{k}$.
Thus our lower bound matches the upper bound  established in \citep{hollender2023computational}, up to a logarithmic factor.
The proof is based on the reduction to the problem of finding local minima in grid graph~\citep{vavasis1993black,bubeck2020trap} and the adaptive complexity of finding local minima in grid graph~\citep{branzei2022query}. The details can be found in Appendix~\ref{app:reduction}.

\section{Conclusions and future works}
In this paper, we make significant progress towards understanding parallelization for non-convex optimization, in terms of both upper and lower bounds.
In particular, when $d = \widetilde{\Omega}\left(\varepsilon^{-(2+2p)/{p}}\right)$, we discover that parallelization cannot accelerate non-convex optimization, leading to a complexity of $\Omega(\varepsilon^{-(1+p)/p})$ for finding a stationary point, which matches the lower bound for one-query-per-round algorithms.
When $d = \Theta(1)$, our algorithm finds approximate stationary point within constant iteration and the asymptotically optimal query requirement which matches our lower bound.
Furthermore we  answer the open problem in~\citep{bubeck2020trap} by showing $\widetilde{\Omega}(\varepsilon^{-(d-1)/2})$ is necessary for $\log(1/\varepsilon)$ iteration algorithms.
%


There are several interesting directions for future exploration. Exploring the role of parallelization for intermediate dimension regime ($ d=  \omega(1) \leq d \leq O(\varepsilon^{-4})$) is an intriguing challenge that continues to be of significant academic interest. Na\"ive adaptation of  lower bounds from high-dimensional cases is non-trivial, as the construction used to hide information from polynomial queries relies on high dimensionality to ensure concentration inequalities. Further, the reduction to finding local minima on a grid graph
depends on constant dimensions, which breaks down when dimension increses.
In addition, the query complexity of high-order methods in the constant-dimensional case remains unknown.

\newpage
\bibliography{references}

\begin{thebibliography}{57}
\providecommand{\natexlab}[1]{#1}
\providecommand{\url}[1]{\texttt{#1}}
\expandafter\ifx\csname urlstyle\endcsname\relax
  \providecommand{\doi}[1]{doi: #1}\else
  \providecommand{\doi}{doi: \begingroup \urlstyle{rm}\Url}\fi

\bibitem[Agarwal et~al.(2017)Agarwal, Agarwal, Assadi, and Khanna]{agarwal2017learning}
Arpit Agarwal, Shivani Agarwal, Sepehr Assadi, and Sanjeev Khanna.
\newblock Learning with limited rounds of adaptivity: Coin tossing, multi-armed bandits, and ranking from pairwise comparisons.
\newblock In \emph{Conference on Learning Theory}, pages 39--75. PMLR, 2017.

\bibitem[Allen-Zhu et~al.(2019)Allen-Zhu, Li, and Song]{allen2019convergence}
Zeyuan Allen-Zhu, Yuanzhi Li, and Zhao Song.
\newblock A convergence theory for deep learning via over-parameterization.
\newblock In \emph{International Conference on Machine Learning}, pages 242--252. PMLR, 2019.

\bibitem[Anari et~al.(2024)Anari, Chewi, and Vuong]{anari24a}
Nima Anari, Sinho Chewi, and Thuy-Duong Vuong.
\newblock Fast parallel sampling under isoperimetry.
\newblock In \emph{Proceedings of Thirty Seventh Conference on Learning Theory}, volume 247 of \emph{Proceedings of Machine Learning Research}, pages 161--185. PMLR, 30 Jun--03 Jul 2024.

\bibitem[Arjevani et~al.(2023)Arjevani, Carmon, Duchi, Foster, Srebro, and Woodworth]{arjevani2023lower}
Yossi Arjevani, Yair Carmon, John~C Duchi, Dylan~J Foster, Nathan Srebro, and Blake Woodworth.
\newblock Lower bounds for non-convex stochastic optimization.
\newblock \emph{Mathematical Programming}, 199\penalty0 (1):\penalty0 165--214, 2023.

\bibitem[Balkanski and Singer(2018{\natexlab{a}})]{balkanski2018adaptive}
Eric Balkanski and Yaron Singer.
\newblock The adaptive complexity of maximizing a submodular function.
\newblock In \emph{Proceedings of the 50th annual ACM SIGACT symposium on theory of computing}, pages 1138--1151, 2018{\natexlab{a}}.

\bibitem[Balkanski and Singer(2018{\natexlab{b}})]{balkanski2018parallelization}
Eric Balkanski and Yaron Singer.
\newblock Parallelization does not accelerate convex optimization: Adaptivity lower bounds for non-smooth convex minimization.
\newblock \emph{arXiv preprint arXiv:1808.03880}, 2018{\natexlab{b}}.

\bibitem[Ball et~al.(1997)]{ball1997elementary}
Keith Ball et~al.
\newblock An elementary introduction to modern convex geometry.
\newblock \emph{Flavors of geometry}, 31\penalty0 (1-58):\penalty0 26, 1997.

\bibitem[Birgin et~al.(2017)Birgin, Gardenghi, Mart{\'\i}nez, Santos, and Toint]{birgin2017worst}
Ernesto~G Birgin, JL~Gardenghi, Jos{\'e}~Mario Mart{\'\i}nez, Sandra~Augusta Santos, and Ph~L Toint.
\newblock Worst-case evaluation complexity for unconstrained nonlinear optimization using high-order regularized models.
\newblock \emph{Mathematical Programming}, 163:\penalty0 359--368, 2017.

\bibitem[Bottou et~al.(2018)Bottou, Curtis, and Nocedal]{bottou2018optimization}
L{\'e}on Bottou, Frank~E Curtis, and Jorge Nocedal.
\newblock Optimization methods for large-scale machine learning.
\newblock \emph{SIAM review}, 60\penalty0 (2):\penalty0 223--311, 2018.

\bibitem[Br{\^a}nzei and Li(2022)]{branzei2022query}
Simina Br{\^a}nzei and Jiawei Li.
\newblock The query complexity of local search and brouwer in rounds.
\newblock In \emph{Conference on Learning Theory}, pages 5128--5145. PMLR, 2022.

\bibitem[Braverman et~al.(2016)Braverman, Mao, and Weinberg]{braverman2016parallel}
Mark Braverman, Jieming Mao, and S~Matthew Weinberg.
\newblock Parallel algorithms for select and partition with noisy comparisons.
\newblock In \emph{Proceedings of the forty-eighth annual ACM symposium on Theory of Computing}, pages 851--862, 2016.

\bibitem[Bubeck and Mikulincer(2020)]{bubeck2020trap}
S{\'e}bastien Bubeck and Dan Mikulincer.
\newblock How to trap a gradient flow.
\newblock In \emph{Conference on Learning Theory}, pages 940--960. PMLR, 2020.

\bibitem[Bubeck et~al.(2019)Bubeck, Jiang, Lee, Li, and Sidford]{bubeck2019complexity}
S{\'e}bastien Bubeck, Qijia Jiang, Yin-Tat Lee, Yuanzhi Li, and Aaron Sidford.
\newblock Complexity of highly parallel non-smooth convex optimization.
\newblock \emph{Advances in neural information processing systems}, 32, 2019.

\bibitem[Canonne and Gur(2018)]{canonne2018adaptivity}
Cl{\'e}ment~L Canonne and Tom Gur.
\newblock An adaptivity hierarchy theorem for property testing.
\newblock \emph{computational complexity}, 27:\penalty0 671--716, 2018.

\bibitem[Carmon et~al.(2020)Carmon, Duchi, Hinder, and Sidford]{carmon2020lower}
Yair Carmon, John~C Duchi, Oliver Hinder, and Aaron Sidford.
\newblock Lower bounds for finding stationary points i.
\newblock \emph{Mathematical Programming}, 184\penalty0 (1):\penalty0 71--120, 2020.

\bibitem[Carmon et~al.(2021)Carmon, Duchi, Hinder, and Sidford]{carmon2021lower}
Yair Carmon, John~C Duchi, Oliver Hinder, and Aaron Sidford.
\newblock Lower bounds for finding stationary points ii: first-order methods.
\newblock \emph{Mathematical Programming}, 185\penalty0 (1):\penalty0 315--355, 2021.

\bibitem[Carmon et~al.(2023)Carmon, Jambulapati, Jin, Lee, Liu, Sidford, and Tian]{carmon2023resqueing}
Yair Carmon, Arun Jambulapati, Yujia Jin, Yin~Tat Lee, Daogao Liu, Aaron Sidford, and Kevin Tian.
\newblock Resqueing parallel and private stochastic convex optimization.
\newblock In \emph{2023 IEEE 64th Annual Symposium on Foundations of Computer Science (FOCS)}, pages 2031--2058. IEEE, 2023.

\bibitem[Cartis and Roberts(2023)]{cartis2023scalable}
Coralia Cartis and Lindon Roberts.
\newblock Scalable subspace methods for derivative-free nonlinear least-squares optimization.
\newblock \emph{Mathematical Programming}, 199\penalty0 (1):\penalty0 461--524, 2023.

\bibitem[Cartis et~al.(2010)Cartis, Gould, and Toint]{cartis2010complexity}
Coralia Cartis, Nicholas~IM Gould, and Ph~L Toint.
\newblock On the complexity of steepest descent, newton's and regularized newton's methods for nonconvex unconstrained optimization problems.
\newblock \emph{Siam journal on optimization}, 20\penalty0 (6):\penalty0 2833--2852, 2010.

\bibitem[Cartis et~al.(2020{\natexlab{a}})Cartis, Gould, and Toint]{cartis2020sharp}
Coralia Cartis, Nicholas~IM Gould, and Philippe~L Toint.
\newblock Sharp worst-case evaluation complexity bounds for arbitrary-order nonconvex optimization with inexpensive constraints.
\newblock \emph{SIAM Journal on Optimization}, 30\penalty0 (1):\penalty0 513--541, 2020{\natexlab{a}}.

\bibitem[Cartis et~al.(2020{\natexlab{b}})Cartis, Gould, and Toint]{cartis2020concise}
Coralia Cartis, Nick~IM Gould, and Ph~L Toint.
\newblock A concise second-order complexity analysis for unconstrained optimization using high-order regularized models.
\newblock \emph{Optimization Methods and Software}, 35\penalty0 (2):\penalty0 243--256, 2020{\natexlab{b}}.

\bibitem[Chakrabarty et~al.(2022)Chakrabarty, Graur, Jiang, and Sidford]{chakrabarty2022improved}
Deeparnab Chakrabarty, Andrei Graur, Haotian Jiang, and Aaron Sidford.
\newblock Improved lower bounds for submodular function minimization.
\newblock In \emph{2022 IEEE 63rd Annual Symposium on Foundations of Computer Science (FOCS)}, pages 245--254. IEEE, 2022.

\bibitem[Chakrabarty et~al.(2024)Chakrabarty, Graur, Jiang, and Sidford]{chakrabarty2024parallel}
Deeparnab Chakrabarty, Andrei Graur, Haotian Jiang, and Aaron Sidford.
\newblock Parallel submodular function minimization.
\newblock \emph{Advances in Neural Information Processing Systems}, 36, 2024.

\bibitem[Chen et~al.(2018)Chen, Servedio, Tan, Waingarten, and Xie]{chen2018settling}
Xi~Chen, Rocco~A Servedio, Li-Yang Tan, Erik Waingarten, and Jinyu Xie.
\newblock Settling the query complexity of non-adaptive junta testing.
\newblock \emph{Journal of the ACM (JACM)}, 65\penalty0 (6):\penalty0 1--18, 2018.

\bibitem[Choromanska et~al.(2015)Choromanska, Henaff, Mathieu, Arous, and LeCun]{choromanska2015loss}
Anna Choromanska, Mikael Henaff, Michael Mathieu, G{\'e}rard~Ben Arous, and Yann LeCun.
\newblock The loss surfaces of multilayer networks.
\newblock In \emph{Artificial Intelligence and Statistics}, pages 192--204. PMLR, 2015.

\bibitem[Cole(1988)]{cole1988parallel}
Richard Cole.
\newblock Parallel merge sort.
\newblock \emph{SIAM Journal on Computing}, 17\penalty0 (4):\penalty0 770--785, 1988.

\bibitem[Dean et~al.(2012)Dean, Corrado, Monga, Chen, Devin, Mao, Ranzato, Senior, Tucker, Yang, et~al.]{dean2012large}
Jeffrey Dean, Greg Corrado, Rajat Monga, Kai Chen, Matthieu Devin, Mark Mao, Marc'aurelio Ranzato, Andrew Senior, Paul Tucker, Ke~Yang, et~al.
\newblock Large scale distributed deep networks.
\newblock \emph{Advances in Neural Information Processing Systems}, 25, 2012.

\bibitem[Diakonikolas and Guzm{\'a}n(2019)]{diakonikolas2019lower}
Jelena Diakonikolas and Crist{\'o}bal Guzm{\'a}n.
\newblock Lower bounds for parallel and randomized convex optimization.
\newblock In \emph{Conference on Learning Theory}, pages 1132--1157. PMLR, 2019.

\bibitem[Fang et~al.(2018)Fang, Li, Lin, and Zhang]{fang2018spider}
Cong Fang, Chris~Junchi Li, Zhouchen Lin, and Tong Zhang.
\newblock Spider: Near-optimal non-convex optimization via stochastic path-integrated differential estimator.
\newblock \emph{Advances in Neural Information Processing Systems}, 31, 2018.

\bibitem[Ge et~al.(2015)Ge, Huang, Jin, and Yuan]{ge2015escaping}
Rong Ge, Furong Huang, Chi Jin, and Yang Yuan.
\newblock Escaping from saddle points—online stochastic gradient for tensor decomposition.
\newblock In \emph{Conference on learning theory}, pages 797--842. PMLR, 2015.

\bibitem[Ge et~al.(2016)Ge, Lee, and Ma]{ge2016matrix}
Rong Ge, Jason~D Lee, and Tengyu Ma.
\newblock Matrix completion has no spurious local minimum.
\newblock \emph{Advances in neural information processing systems}, 29, 2016.

\bibitem[Hollender and Zampetakis(2023)]{hollender2023computational}
Alexandros Hollender and Emmanouil Zampetakis.
\newblock The computational complexity of finding stationary points in non-convex optimization.
\newblock In \emph{The Thirty Sixth Annual Conference on Learning Theory}, pages 5571--5572. PMLR, 2023.

\bibitem[Jain et~al.(2017)Jain, Kar, et~al.]{jain2017non}
Prateek Jain, Purushottam Kar, et~al.
\newblock Non-convex optimization for machine learning.
\newblock \emph{Foundations and Trends{\textregistered} in Machine Learning}, 10\penalty0 (3-4):\penalty0 142--363, 2017.

\bibitem[Jin et~al.(2017)Jin, Ge, Netrapalli, Kakade, and Jordan]{jin2017escape}
Chi Jin, Rong Ge, Praneeth Netrapalli, Sham~M Kakade, and Michael~I Jordan.
\newblock How to escape saddle points efficiently.
\newblock In \emph{International Conference on Machine Learning}, pages 1724--1732. PMLR, 2017.

\bibitem[Kawaguchi(2016)]{kawaguchi2016deep}
Kenji Kawaguchi.
\newblock Deep learning without poor local minima.
\newblock \emph{Advances in Neural Information Processing Systems}, 29, 2016.

\bibitem[Kingma(2014)]{kingma2014adam}
Diederik~P Kingma.
\newblock Adam: A method for stochastic optimization.
\newblock \emph{arXiv preprint arXiv:1412.6980}, 2014.

\bibitem[Kwon et~al.(2024)Kwon, Kwon, and Lyu]{kwon2024complexity}
Jeongyeol Kwon, Dohyun Kwon, and Hanbaek Lyu.
\newblock On the complexity of first-order methods in stochastic bilevel optimization.
\newblock \emph{arXiv preprint arXiv:2402.07101}, 2024.

\bibitem[Li et~al.(2020)Li, Liu, and Vondr{\'a}k]{li2020polynomial}
Wenzheng Li, Paul Liu, and Jan Vondr{\'a}k.
\newblock A polynomial lower bound on adaptive complexity of submodular maximization.
\newblock In \emph{Proceedings of the 52nd Annual ACM SIGACT Symposium on Theory of Computing}, pages 140--152, 2020.

\bibitem[Ma et~al.(2018)Ma, Wang, Chi, and Chen]{ma2018implicit}
Cong Ma, Kaizheng Wang, Yuejie Chi, and Yuxin Chen.
\newblock Implicit regularization in nonconvex statistical estimation: Gradient descent converges linearly for phase retrieval and matrix completion.
\newblock In \emph{International Conference on Machine Learning}, pages 3345--3354. PMLR, 2018.

\bibitem[Murty and Kabadi(1985)]{murty1985some}
Katta~G Murty and Santosh~N Kabadi.
\newblock Some np-complete problems in quadratic and nonlinear programming.
\newblock Technical report, 1985.

\bibitem[Nemirovskij and Yudin(1983)]{nemirovskij1983problem}
Arkadij~Semenovi{\v{c}} Nemirovskij and David~Borisovich Yudin.
\newblock Problem complexity and method efficiency in optimization.
\newblock 1983.

\bibitem[Nesterov(2012)]{nesterov2012make}
Yurii Nesterov.
\newblock How to make the gradients small.
\newblock \emph{Optima. Mathematical Optimization Society Newsletter}, \penalty0 (88):\penalty0 10--11, 2012.

\bibitem[Nesterov(2013)]{nesterov2013introductory}
Yurii Nesterov.
\newblock \emph{Introductory lectures on convex optimization: A basic course}, volume~87.
\newblock Springer Science \& Business Media, 2013.

\bibitem[Nesterov and Polyak(2006)]{nesterov2006cubic}
Yurii Nesterov and Boris~T Polyak.
\newblock Cubic regularization of newton method and its global performance.
\newblock \emph{Mathematical programming}, 108\penalty0 (1):\penalty0 177--205, 2006.

\bibitem[Polaczyk(2023)]{polaczyk2023concentration}
Bart{\l}omiej Polaczyk.
\newblock \emph{Concentration of Measure and Functional Inequalities}.
\newblock PhD thesis, University of Warsaw, 2023.

\bibitem[Recht et~al.(2011)Recht, Re, Wright, and Niu]{recht2011hogwild}
Benjamin Recht, Christopher Re, Stephen Wright, and Feng Niu.
\newblock Hogwild!: A lock-free approach to parallelizing stochastic gradient descent.
\newblock \emph{Advances in Neural Information Processing Systems}, 24, 2011.

\bibitem[Sun et~al.(2018)Sun, Qu, and Wright]{sun2018geometric}
Ju~Sun, Qing Qu, and John Wright.
\newblock A geometric analysis of phase retrieval.
\newblock \emph{Foundations of Computational Mathematics}, 18:\penalty0 1131--1198, 2018.

\bibitem[Valiant(1975)]{valiant1975parallelism}
Leslie~G Valiant.
\newblock Parallelism in comparison problems.
\newblock \emph{SIAM Journal on Computing}, 4\penalty0 (3):\penalty0 348--355, 1975.

\bibitem[Vavasis(1993)]{vavasis1993black}
Stephen~A Vavasis.
\newblock Black-box complexity of local minimization.
\newblock \emph{SIAM Journal on Optimization}, 3\penalty0 (1):\penalty0 60--80, 1993.

\bibitem[Woodworth and Srebro(2017)]{woodworth2017lower}
Blake Woodworth and Nathan Srebro.
\newblock Lower bound for randomized first order convex optimization.
\newblock \emph{arXiv preprint arXiv:1709.03594}, 2017.

\bibitem[You et~al.(2017)You, Gitman, and Ginsburg]{you2017large}
Yang You, Igor Gitman, and Boris Ginsburg.
\newblock Large batch training of convolutional networks.
\newblock \emph{arXiv preprint arXiv:1708.03888}, 2017.

\bibitem[You et~al.(2020)You, Li, Reddi, Hseu, Kumar, Bhojanapalli, Song, Demmel, Keutzer, and Hsieh]{You2020Large}
Yang You, Jing Li, Sashank Reddi, Jonathan Hseu, Sanjiv Kumar, Srinadh Bhojanapalli, Xiaodan Song, James Demmel, Kurt Keutzer, and Cho-Jui Hsieh.
\newblock Large batch optimization for deep learning: Training {BERT} in 76 minutes.
\newblock In \emph{International Conference on Learning Representations}, 2020.
\newblock URL \url{https://openreview.net/forum?id=Syx4wnEtvH}.

\bibitem[Yue et~al.(2023)Yue, Fang, and Lin]{yue2023lower}
Pengyun Yue, Cong Fang, and Zhouchen Lin.
\newblock On the lower bound of minimizing polyak-{\l}ojasiewicz functions.
\newblock In \emph{The Thirty Sixth Annual Conference on Learning Theory}, pages 2948--2968. PMLR, 2023.

\bibitem[Zaheer et~al.(2018)Zaheer, Reddi, Sachan, Kale, and Kumar]{zaheer2018adaptive}
Manzil Zaheer, Sashank Reddi, Devendra Sachan, Satyen Kale, and Sanjiv Kumar.
\newblock Adaptive methods for nonconvex optimization.
\newblock \emph{Advances in Neural Information Processing systems}, 31, 2018.

\bibitem[Zhang et~al.(2022)Zhang, Hong, and Zhang]{zhang2022lower}
Junyu Zhang, Mingyi Hong, and Shuzhong Zhang.
\newblock On lower iteration complexity bounds for the convex concave saddle point problems.
\newblock \emph{Mathematical Programming}, 194\penalty0 (1):\penalty0 901--935, 2022.

\bibitem[Zhou and Sugiyama(2024)]{zhou2024parallel}
Huanjian Zhou and Masashi Sugiyama.
\newblock Parallel simulation for sampling under isoperimetry and score-based diffusion models.
\newblock \emph{arXiv preprint arXiv:2412.07435}, 2024.

\bibitem[Zhou et~al.(2024)Zhou, Wang, and Sugiyama]{zhou2024adaptive}
Huanjian Zhou, Baoxiang Wang, and Masashi Sugiyama.
\newblock Adaptive complexity of log-concave sampling.
\newblock \emph{arXiv preprint arXiv:2408.13045}, 2024.

\end{thebibliography}

\appendix

\newpage

\section{Useful facts}

\begin{theorem}[\textbf{Concentration of linear function of conditioned Bernoullis~(Theorem 4.2.5 in \cite{polaczyk2023concentration})}]
\label{theo:concentration2}
Let $X_1, \ldots, X_n$ be $\{0, 1\}$ random variables conditioned on $\sum\limits_{i=1}^n X_i =k$. Let $f : \{0, 1\}^n \to \mathbb{R}$ be $f(\boldsymbol{x}) = \sum\limits_{i=1}^n\alpha_i\boldsymbol{x}_i$ with $\alpha_i\geq 0$ for all $i\in [n]$. Then for any
$t > 0$,
\[\mathbb{P}[|f(X_1, \ldots, X_n)- \mathbb{E}[f(X_1, \ldots, X_n)]| \geq  t] \leq 2\exp\left(-\frac{t^2}{16\sum\limits_{i=1}^k(\alpha_i^{\downarrow})^2}\right),\]
where for a finite sequence $x$, we denote by $x^\downarrow$ the non-increasing rearrangement of the elements of $x$.
\end{theorem}


\section{Missing proof in Section \ref{sec:high}}

\subsection{Proof of Lemma \ref{lmm:g}}
\label{app:g}
\begin{lemma}[\textbf{Smoothness and boundness of $g_\mathcal{P}$}]
The function $g_\mathcal{P}$ satisfies the following.
\begin{enumerate}
\item We have $g_\mathcal{P}(\boldsymbol{0})-\inf_{\boldsymbol{x}}g_\mathcal{P}(\boldsymbol{x})< 12 r $;
\item For all $\boldsymbol{x}\in \mathbb{R}^d$, $\left\|\nabla g_\mathcal{P}(\boldsymbol{x}) \right\|\leq 46\sqrt{r+1}$.
\end{enumerate}
\end{lemma}
\begin{proof}
Part 1 follows because $g_\mathcal{P}(\boldsymbol{0}) = -\Psi(1)\Phi(0)<0$ and since $0<\Psi<e$ and $0\leq \Phi\leq \sqrt{2\pi e}$, we have 
\[g_\mathcal{P}(\boldsymbol{x})\geq -re\sqrt{2\pi e}\geq -12 r.\]
For part 2, we first bound the partial derivative as follows. If $l\in P_1$, we have
\begin{align*}
&\sqrt{d_0}|\partial_l g_\mathcal{P}(\boldsymbol{y})|\\~=~& \left|-\Psi(1)\Phi'(X^1 ) 
- \Psi'(-X^1)\Phi(X^2-X^1) - \Psi(-X^1)\Phi'(X^2-X^1)\right.\\
&-\Psi'(X^1)\Phi(X^1-X^2) - \Psi(X^1)\Phi'(X^1-X^2) \\
&  \left.- \Psi'(X^1- X^2)\Phi(X^3 - X^2) 
- \Psi'(-X^1+ X^2)\Phi(X^2 - X^3)  \right|\\
~\leq~&2e\sqrt{e} + 2\sqrt{54/e} \sqrt{2\pi e}\leq 46,
\end{align*}
where the last inequality holds since $\Psi(x) = 0$ and $\Psi'(x) = 0$ on any $x < 1/2$,  $0 \le
  \Psi'(x) \le \sqrt{54e^{-1}}$ and $0 \le \Phi'(x) \le
  \sqrt{e}$.
If $l\in P_j$ and $1<j<r+1$, we have
\begin{align*}
&\sqrt{d_0}|\partial_l g_\mathcal{P}(\boldsymbol{y})|\\
~=~&  \left|-(-1)^{j+1}\left[\Psi'(X^{j}- X^{j+1})\Phi(X^{j+2}- X^{j+1}) 
 + \Psi'(X^{j+1}- X^{j})\Phi(X^{j+1} - X^{j+2}) \right]\right.
 \\
 &-(-1)^{j} \left[- \Psi'(X^{j-1}- X^j)\Phi(X^{j+1}-X^{j})- \Psi(X^{j-1}- X^j)\Phi'(X^{j+1}-X^{j})\right.
 \\&
 \left.-\Psi'(X^j - X^{j-1})\Phi(X^j-X^{j+1}) -\Psi(X^j - X^{j-1})\Phi'(X^j-X^{j+1})\right] \\&
 \left.-(-1)^{j-1}\left[\Psi(X^{j-2}- X^{j-1})\Phi'(X^{j}- X^{j-1}) + \Psi(X^{j-1}- X^{j-2})\Phi'(X^{j-1}-X^{j})\right] \right| \\
 ~=~& \left|\Psi'(X^{j}- X^{j+1})\Phi(X^{j+2}- X^{j+1}) 
 + \Psi'(X^{j+1}- X^{j})\Phi(X^{j+1} - X^{j+2}) \right.
 \\
 & + \Psi'(X^{j-1}- X^j)\Phi(X^{j+1}-X^{j})+ \Psi(X^{j-1}- X^j)\Phi'(X^{j+1}-X^{j})
 \\&
+\Psi'(X^j - X^{j-1})\Phi(X^j-X^{j+1}) +\Psi(X^j - X^{j-1})\Phi'(X^j-X^{j+1}) \\&
+\left.\Psi(X^{j-2}- X^{j-1})\Phi'(X^{j}- X^{j-1}) + \Psi(X^{j-1}- X^{j-2})\Phi'(X^{j-1}-X^{j})\right|\\
~\leq~& 2\sqrt{54/e}\sqrt{2\pi e}+ 2e\sqrt{e}\leq 46.
\end{align*}
If $l \in P_{r+1}$, we have 
\begin{align*}
\sqrt{d_0}|\partial_l g_\mathcal{P}(\boldsymbol{y})|
~=~&  \left[\Psi(X^{r-1}- X^{r})\Phi'(X^{r+1}- X^{r}) 
 + \Psi(X^{r}- X^{r-1})\Phi'(X^{r} - X^{r+1}) \right]\leq 2e\sqrt{e}\leq 23.
\end{align*}
Consequently,
\[\left\|\nabla g_\mathcal{P}(\boldsymbol{x}) \right\|^2 \leq r\frac{1}{d_0}d_0 46^2 +\frac{1}{d_0}d_0 23^2\leq 46^2(r+1).\]
\end{proof}

\subsection{Proof of Lemma~\ref{lmm:f_p}}
\label{app:f}
\begin{lemma}[\textbf{Smoothness and boundness of $f_\mathcal{P}$}]
The function $f_\mathcal{P}$ satisfies the following,
\begin{enumerate}
\item We have $f_\mathcal{P}(\boldsymbol{0}) - \inf_{\boldsymbol{x}}f_\mathcal{P}(\boldsymbol{x})<12r$.
\item For every $p \geq 1$, the $p$-th order derivatives of $f_\mathcal{P}$ are $l_p$-Lipschitz continuous, where $l_p \leq
\exp(cp \log p + c)$ for a numerical constant $c < \infty$.
\end{enumerate}
\end{lemma}
\begin{proof}
For the first part, we have $f_\mathcal{P}(\boldsymbol{0}) = g_\mathcal{P}(\boldsymbol{0})<0$ and 
\[\inf\limits_{\boldsymbol{x}\in \mathbb{R}^d}f_\mathcal{P}(\boldsymbol{x})\geq \inf\limits_{\boldsymbol{x}\in \mathbb{R}^d}g_\mathcal{P}(\rho(\boldsymbol{x}))\geq\inf\limits_{\left\|\boldsymbol{x}\right\|\leq R}g_\mathcal{P}(\boldsymbol{x})\geq \inf\limits_{\boldsymbol{x}\in \mathbb{R}^d}g_\mathcal{P}(\boldsymbol{x})\geq -12r. \]
For the second part, we first consider $\widetilde{g}:\mathbb{R}^d \to \mathbb{R}$ defined as
\[\widetilde{g}(\boldsymbol{x}) =  -\Psi(1)\Phi(\boldsymbol{x}_1)- \sum\limits_{i=1}^{r}(-1)^i\left[\Psi(-\boldsymbol{x}_{i})\Phi(\boldsymbol{x}_{i+1}) - \Psi(\boldsymbol{x}_{i})\Phi(-\boldsymbol{x}_{i+1}) \right].\]
Then $g_\mathcal{P}(\boldsymbol{x}) = \widetilde{g}(PM\boldsymbol{x})$, where $M_p\in \mathbb{R}^{d\times d}$ is a permutation matrix respecting $\mathcal{P}$, such that for any $\boldsymbol{x}$ we have $M \boldsymbol{x}=[\boldsymbol{x}_{P_1},\ldots,\boldsymbol{x}_{P_{r+1}}]^\top$ and $P\in \mathbb{R}^{d\times d}$ is a matrix for $X^i-X^{i-1}$ such that $P_{i,j} = \begin{cases}
    1/\sqrt{d_0} & i\in [r+1] \mbox{ and } (i-1)d_0+1\leq j\leq id_0,\\
    -1/\sqrt{d_0} & 2\leq i\leq r+1 \mbox{ and } (i-2)d_0+1\leq j\leq (i-1)d_0,\\
    0 & \mbox{otherwise}.
\end{cases}$ 

We first bound the norm $(p+1)$-th order derivatives of $\widetilde{g}$. To do so, we first drop the coordinates from $r+2$ to $d$ for $\widetilde{g}$, then fix a point $\boldsymbol{x} \in \mathbb{R}^{r+1}$ and a unit vector $\boldsymbol{v} \in \mathbb{R}^{r+1}$. We define the function $h_{\boldsymbol{x},\boldsymbol{v}}:\mathbb{R}\to \mathbb{R}$ as $h_{\boldsymbol{x},\boldsymbol{v}}(\theta) = \widetilde{g}(\boldsymbol{x} + \theta\boldsymbol{v})$. The function $\theta \mapsto h_{\boldsymbol{x},\boldsymbol{v}}(\theta)$
  is infinitely differentiable for every $\boldsymbol{x}$ and $\boldsymbol{v}$. Therefore, $\widetilde{g}$
  has $l_p$-Lipschitz $p$-th order derivatives if and only if
  $|h^{(p+1)}_{\boldsymbol{x},\boldsymbol{v}}(0)| \le l_p$ for every $\boldsymbol{x}$, $\boldsymbol{v}$. Using the
  shorthand notation $\partial_{i_1}\cdots \partial_{i_k}$ for $\frac{\partial^k}{\partial
    \boldsymbol{x}_{i_1} \cdots \partial \boldsymbol{x}_{i_k}}$, we have
\[h^{(p+1)}_{\boldsymbol{x},\boldsymbol{v}}(0) = \sum_{i_{1}, \ldots, i_{p+1}=1}^{r+1}\partial_{i_{1}}\cdots\partial_{i_{p+1}}\widetilde{g}\left(\boldsymbol{x}\right)\boldsymbol{v}_{i_{1}}\cdots \boldsymbol{v}_{i_{p+1}}.\]
 Observing that $\partial_{i_{1}}\cdots\partial_{i_{p+1}}\widetilde{g}$
  is non-zero if and only if $\left|i_{j}-i_{k}\right|\le1$ for every
  $j,k\in\left[p+1\right]$. 
Consequently, we can rearrange the above summation as
\begin{equation}
\label{eq:11}
h^{(p+1)}_{\boldsymbol{x},\boldsymbol{v}}(0) = \sum_{\delta_{1},\delta_{2},\ldots,\delta_{p}\in\left\{ 0,1\right\}^{p}
      \cup \left\{ 0,-1\right\}^{p}}
\sum_{i=1}^{r+1}\partial_{i+\delta_{1}}\cdots\partial_{i+\delta_{p}}\partial_{i}\widetilde{g}\left(\boldsymbol{x}\right)\boldsymbol{v}_{i+\delta_{1}}\cdots \boldsymbol{v}_{i+\delta_{p}}\boldsymbol{v}_{i},    
\end{equation}
where we let $\boldsymbol{v}_0 = \boldsymbol{v}_{r+2} = 0$.
Recall that for all $x \le 1/2$ and
    all $k \in [N]$, $\Psi^{(k)}(x) = 0$, $ 0 \le \Psi,\Phi < \sqrt{2\pi e},$ and for all $k \in [N]$ we have
    \begin{equation*}
      \sup_x |\Psi^{(k)}(x)|
      \le \exp\left(\frac{5 k}{2}\log(4 k)\right)
      ~~\mbox{and}~~
      \sup_x |\Phi^{(k)}(x)|
      \le \exp\left(\frac{3k}{2} \log \frac{3k}{2} \right).
    \end{equation*}
Thus we have 
\begin{align}
\label{eq:12}
\sup\limits_{\boldsymbol{x} \in \R^{r+1}}  &
    \max\limits_{i \in [r+1]} \max\limits_{\delta \in \{0, 1\}^p
      \cup \{0, -1\}^p}
      \left|\partial_{i+\delta_{1}}\cdots\partial_{i+\delta_{p}}\partial_{i}\widetilde{g}\left(\boldsymbol{x}\right)\right|  
    \le \max\limits_{k\in[p+1]}  \left\{ 2\sup\limits_{x\in\R} \left| \Psi^{(k)}(x)\right| \sup\limits_{x'\in\R} \left| \Phi^{(p+1-k)}(x')\right| \right\}\nonumber
    \\ & \le 
    2\sqrt{2\pi e}\cdot e^{2.5(p+1)\log(4 (p+1))} \le 
    \exp\left(2.5p\log p + 4p + 9\right). 
\end{align}
We further claim that $\left|\sum\limits_{i=1}^{r}\boldsymbol{v}_{i+\delta_{1}}\cdots\boldsymbol{v}_{i+\delta_{p}}\boldsymbol{v}_{i}\right|\le1$ for every $\delta \in \{0, 1\}^p \cup \{0,
  -1\}^p$. 
When $\delta = \boldsymbol{0}$, we  have $\left|\sum\limits_{i=1}^{r+1}\boldsymbol{v}_{i+\delta_{1}}\cdots\boldsymbol{v}_{i+\delta_{p}}\boldsymbol{v}_{i}\right|\leq \left|\sum\limits_{i=1}^{r+1}\boldsymbol{v}_{i}^p\right|\leq \left|\sum\limits_{i=1}^{r+1}\boldsymbol{v}_{i}^2\right|= 1$. Otherwise, let $1\le\sum_{j=1}^{p}|\delta_{j}|=n\le p$, by the Cauchy-Swartz inequality, we have, for $s\in \{\pm 1\}$,
\[\left|\sum\limits_{i=1}^{r}\boldsymbol{v}_{i+\delta_{1}}\cdots\boldsymbol{v}_{i+\delta_{p}}\boldsymbol{v}_{i}\right| = \left|\sum_{i=1}^{r+1} \boldsymbol{v}_{i}^{p + 1- n}
    \boldsymbol{v}_{i+s}^{n}\right|\leq \sqrt{\sum_{i=1}^{r+1}\boldsymbol{v}_{i}^{2\left(p+1-n\right)}}
\sqrt{\sum_{i=1}^{r+1}\boldsymbol{v}_{i+s}^{2n}}\le\sum_{i=1}^{r+1}v_{i}^{2}=1.\]
Combining Eq.~\eqref{eq:11}, and Eq.~\eqref{eq:12}, we have
\begin{align*}
h^{(p+1)}_{\boldsymbol{x},\boldsymbol{v}}(0) ~=~&  \sum_{\delta_{1},\delta_{2},\ldots,\delta_{p}\in\left\{ 0,1\right\}^{p}
      \cup \left\{ 0,-1\right\}^{p}}
\sum_{i=1}^{r+1}\partial_{i+\delta_{1}}\cdots\partial_{i+\delta_{p}}\partial_{i}\widetilde{g}\left(\boldsymbol{x}\right)\boldsymbol{v}_{i+\delta_{1}}\cdots \boldsymbol{v}_{i+\delta_{p}}\boldsymbol{v}_{i}    \\
~\leq~&\sum_{\delta_{1},\delta_{2},\ldots,\delta_{p}\in\left\{ 0,1\right\}^{p}
      \cup \left\{ 0,-1\right\}^{p}}  \exp\left(2.5p\log p + 4p + 9\right)
\sum_{i=1}^{r+1}\boldsymbol{v}_{i+\delta_{1}}\cdots \boldsymbol{v}_{i+\delta_{p}}\boldsymbol{v}_{i}    \\
~\leq~&\sum_{\delta_{1},\delta_{2},\ldots,\delta_{p}\in\left\{ 0,1\right\}^{p}
      \cup \left\{ 0,-1\right\}^{p}}  \exp\left(2.5p\log p + 4p + 9\right)  \\
~\leq~&(2^{p+1}-1) \exp\left(2.5p\log p + 4p + 9\right) \leq e^{2.5p\log p + 5p + 10}.
\end{align*}
We further bound the operator norm of $PM\in \mathbb{R}^{d\times d}$ as
\[\left\|PM\right\|_{\mathrm{op}}\leq \left\|P\right\|_{\mathrm{op}}\left\|M\right\|_{\mathrm{op}}\leq\sup\limits_{\boldsymbol{v}:\left\|\boldsymbol{v}\right\|=1}\left\|A\boldsymbol{v}\right\| \leq \sum\limits_{i=1}^{2d_0} \frac{1}{\sqrt{d_0}} \frac{1}{\sqrt{2d_0}} = \sqrt{2}\leq 2. \]
Thus $\left\|\nabla^{(p+1)}g_\mathcal{P}\right\|_{\mathrm{op}} \leq \left\|\nabla^{(p+1)}\widetilde{g}\right\|_{\mathrm{op}}\left\|PM\right\|_{\mathrm{op}}\leq e^{2.5p\log p + 5p + 10} 2^{p+1}\leq e^{2.5p\log p + 6p + 11}$. 

Finally, we bound $f_\mathcal{P} = g_\mathcal{P}\circ \rho +\frac{1}{5}\left\|\cdot\right\|^2$. Since $\frac{1}{5}\left\|\cdot\right\|^2$ is $2/5$-Lipschitz first derivative and $0$-Lipschitz higher order derivatives, we only consider $\widetilde{f} = g_\mathcal{P}\circ \rho$.

To apply Fa\`{a} di Bruno formula to calculate high derivatives of  $\widetilde{f}$, we first define $\mathcal{P}_k$ to be the
  set of all partitions of $[k] = \{1, \ldots, k\}$, \emph{i.e.,} $(S_1, \ldots,
  S_L) \in \mathcal{P}_k$ if and only if the $S_i$ are disjoint and $\cup_l S_l =
  [k]$. Then we have 
\[\nabla^k{i_1,\ldots,i_k}\big[\widetilde{f}~\big] =  \sum\limits_{(S_1,\ldots,S_L)\in \mathcal{P}_k}\sum\limits_{j_1,\ldots,j_L=1}^{r+1}\left(\prod\limits_{l=1}^L\nabla_{i_{S_l}}^{|S_l|}\left[\rho_{j_l}\right]\right)\nabla^L_{j_1,\ldots,j_L} \left[g_\mathcal{P}\right]\circ \rho, \]
where we have used the shorthand $\nabla^{|S|}_{i_{S}}$ to denote
  the partial derivatives with respect to each of $x_{i_j}$ for $j \in S$. Let $\boldsymbol{v} \in \R^d$ be an unit
  arbitrary direction vector with $\left\|\boldsymbol{v}\right\| = 1$. Then for any $j \in [d]$ and $k
  \in [N]$ and fixed $\boldsymbol{x}$, we define $\widetilde{\boldsymbol{v}}_j^k
    =  \big\langle \nabla^{k} \left[\rho_j\right](\boldsymbol{x}), \boldsymbol{v}^{\otimes k}\big\rangle$. Then we have 
\begin{align*}
\big\langle\nabla^k\big[\widetilde{f}~\big](\boldsymbol{x}),\boldsymbol{v}^{\otimes k}  \big\rangle ~=~& \sum\limits_{(S_1,\ldots,S_L)\in \mathcal{P}_k}\sum\limits_{i_1,\ldots,i_k=1}^{d} \boldsymbol{v}_{i_1}\ldots\boldsymbol{v}_{i_k}\sum\limits_{j_1,\ldots,j_L=1}^{r+1}\left(\prod\limits_{l=1}^L\nabla_{i_{S_l}}^{|S_l|}\left[\rho_{j_l}\right](\boldsymbol{x})\right)\nabla^L_{j_1,\ldots,j_L} \left[g_\mathcal{P}\right]\circ \rho (\boldsymbol{x})\\
~=~& \sum\limits_{(S_1,\ldots,S_L)\in \mathcal{P}_k}\sum\limits_{j_1,\ldots,j_L=1}^{r+1} \widetilde{\boldsymbol{v}}_{j_1}^{|S_1|} \ldots\widetilde{\boldsymbol{v}}_{j_L}^{|S_L|} \nabla^L_{j_1,\ldots,j_L} \left[g_\mathcal{P}\right]\circ \rho (\boldsymbol{x})\\
~=~& \sum\limits_{(S_1,\ldots,S_L)\in \mathcal{P}_k}
\bigg\langle \nabla^L \left[g_\mathcal{P}\right]\circ \rho (\boldsymbol{x}), \widetilde{\boldsymbol{v}}^{|S_1|} \ldots\widetilde{\boldsymbol{v}}^{|S_L|}\bigg\rangle.
\end{align*}
We then recall the fact about $\widetilde{\boldsymbol{v}}^{k}(\boldsymbol{x})$ shown in \cite[Section B.4]{carmon2020lower} as follows.
\begin{fact}
There exists a numerical constant
  $c < \infty$ such that for all $k \in [N]$,
  \begin{equation*}
\sup\limits_{\boldsymbol{x}}\left\|\widetilde{\boldsymbol{v}}^{k}(\boldsymbol{x})\right\| \le
    \exp(c k \log k + c) R^{1 - k}.
  \end{equation*}
\end{fact}

Then, combining the fact $\left\|\nabla g_\mathcal{P}(\boldsymbol{x}) \right\|\leq 46\sqrt{r+1}$ and $\left\|\nabla^{(p+1)}g_\mathcal{P}\right\|_{\mathrm{op}} \leq e^{2.5p\log p + 6p + 11}$, we have
\begin{align*}
&\left|\big\langle\nabla^{(p+1)}\big[\widetilde{f}~\big](\boldsymbol{x}),\boldsymbol{v}^{\otimes (p+1)}  \big\rangle\right|\\
~\leq~& \sum\limits_{(S_1,\ldots,S_L)\in \mathcal{P}_{p+1}}
\left\| \nabla^L \left[g_\mathcal{P}\right]\circ \rho (\boldsymbol{x})\right\|_{\mathrm{op}}\prod_{l=1}^L\left\| \widetilde{\boldsymbol{v}}^{|S_l|} \right\|\\
~\leq~& \sum\limits_{(S_1,\ldots,S_L)\in \mathcal{P}_{p+1}}
46\sqrt{r+1} e^{2.5p\log p + 6p + 11}
\prod_{l=1}^L \exp(c |S_l| \log |S_l| + c) R^{1 - |S_l|}\\
~\leq~& e^{(p+1) \log (p+1) }
46\sqrt{r+1} e^{2.5p\log p + 6p + 11}
 e^{c (p+1) \log (p+1) + c} (230\sqrt{r+1})^{-p}\\
 ~\leq~& e^{c' p\log p+c'}.
\end{align*}
for a numerical constant $c'<\infty$ and the penultimate inequality follows from the fact that $R = 230\sqrt{r+1}$,  $q(x) = (x+1)\log(x+1)$ satisfies $q(x)+q(y) \le q(x+y)$ for every $x,y>0$ and there are at
  most $\exp(k \log k)$ partitions in $\mathcal{P}_k$.
\end{proof}

\subsection{Proof of Lemma~\ref{lmm:character}}
\label{app:char}

\begin{lemma}[\textbf{Characterization of output}]
If there exits a constant $\alpha= \omega(1)$ such that $8R\sqrt{\frac{\alpha\log d}{d_0}}\leq \frac{1}{2}$, then 
for any randomized algorithm $\mathsf{A}$, any $\tau \leq r$,  and any initial point $\boldsymbol{x}^0$, 
$X(\mathsf{A}[f_\mathcal{P},\boldsymbol{x}^0,\tau]) = (X^1(\rho(\boldsymbol{x})),\ldots,X^{r+1}(\rho(\boldsymbol{x})))$ takes a form as 
\[(x_1,\ldots,x_{2\tau},x_{2\tau+1},x_{2\tau+1},\ldots,x_{2\tau+1}),\]
up to addictive error $1/4$ {for every coordinate} with probability $1-d^{-\omega(1)}$ over $\mathcal{P}$. 
\end{lemma}
\begin{proof}
We fixed $\tau$ and prove the following by induction for $l\in [\tau]$: With high probability, the computation
path of the (deterministic) algorithm $\mathsf{A}$ and the queries it issues in the $l$-th round are determined by $P_{1},\ldots,P_{2l-2}$.

As a first step, we assume the algorithm is deterministic by fixing its random bits and
choose the partition of $\mathcal{P}$ uniformly at random.

To prove the inductive claim, let $\mathcal{E}_l$ denote the event that
for any query $\boldsymbol{x}$ issued by $\mathsf{A}$ in iteration $l$, the answer is in the form $g^l_\mathcal{P}(\rho(\boldsymbol{x})) + \frac{1}{5}\left\|\boldsymbol{x}\right\|^2$ where $g_\mathcal{P}^l:\mathbb{R}^d\to \mathbb{R}$ defined as:
\[ g^l_\mathcal{P}(\boldsymbol{x})= -\Psi(1)\Phi(X^1)- \sum\limits_{i=1}^{2l-1}(-1)^i\left[\Psi(X^{i-1} - X^i)\Phi(  X^{i+1} - X^{i}) - \Psi(X^i-X^{i-1})\Phi( X^{i}-X^{i+1}) \right],
\] 
i.e., $\mathcal{E}_l$ represents the events that  $\forall \boldsymbol{x} \in Q^l$, $f_\mathcal{P}(\boldsymbol{x} )  = g_\mathcal{P}(\rho(\boldsymbol{x})) + \frac{1}{5}\left\|\boldsymbol{x}\right\|^2 = g^l_\mathcal{P}(\rho(\boldsymbol{x})) + \frac{1}{5}\left\|\boldsymbol{x}\right\|^2$.
This implies that the answers depend only on $P_1,\cdots,P_{2l}$.

Since the queries in round $l$ depend only on $P_{1},\ldots,P_{2l-2}$, if $\mathcal{E}_l$ occurs, the entire computation path in round $l$ is determined by $P_{1},\ldots,P_{2l}$. 
By induction, we conclude that if all of $\mathcal{E}_1, \ldots , \mathcal{E}_{l}$ occur, the computation path in round $l$ is determined by $P_{1},\ldots,P_{2l}$.

Now we analyze the conditional probability $\mathbb{P}\left[\mathcal{E}_l\mid \mathcal{E}_1,\ldots,\mathcal{E}_{l-1}\right]$.
Given all of $\mathcal{E}_1,\ldots,\mathcal{E}_{l-1}$ occur so far, we can claim that $Q^l$ is determined by $P_1,\ldots,P_{2l-2}$.
Conditioned on $P_1,\ldots,P_{2l-2}$, the partition of $[d]\setminus \mathop{\bigcup}_{i\in [2l-2]}P_i$ is uniformly random. 
We consider $\{0,1\}$-random variable $Y_j$, $j\in [d]\setminus \mathop{\bigcup}_{i\in [2l-2]}P_i$.
For any query $\boldsymbol{x} \in \mathbb{R}^d$, we represent $X^i(\rho(\boldsymbol{x}))$ as a linear function of $Y_i$s as  $X^i(\rho(\boldsymbol{x})) = \frac{1}{\sqrt{d_0}}\sum\limits_{j\in [d]\setminus \mathop{\bigcup}_{i\in [2l-2]}P_i} Y_j \rho_j(\boldsymbol{x})$ such that $Y_i=1$ if $Y_i \in P_i$ and $Y_i= 0$ otherwise. 
By the {concentration of linear functions over the Boolean slice} (Theorem \ref{theo:concentration2}), and set $t = 8R\sqrt{\frac{\alpha\log d}{d_0}}$ with $\alpha= \omega(1)$, we have
\begin{align*}
 \mathbb{P}_{\mathcal{P}}\Big[|X^i(\rho(\boldsymbol{x}))- \mathbb{E}[X^i(\rho(\boldsymbol{x}))]| \geq  \frac{t}{2}\Big]
 ~\leq~& 2\exp\Big(-\frac{t^2}{16\sum\nolimits_i\rho^2_i(\boldsymbol{x})/d_0}\Big)
=2d^{-\omega(1)}.   
\end{align*}
Similarly, $\mathbb{P}_{\mathcal{P}}\left[|X^{i+1}(\rho(\boldsymbol{x}))- \mathbb{E}[X^{i+1}(\rho(\boldsymbol{x}))]| \geq  \frac{t}{2}\right] \leq  2d^{-\omega(1)}$. 
Combining the fact that $\mathbb{E}[X^i(\rho(\boldsymbol{x}))] = \mathbb{E}[X^{i+1}(\rho(\boldsymbol{x}))]$, when $t\leq \frac{1}{2} $ we have with probability at least  $1-d^{-\omega(1)}$, for any fixed $i\geq 2l$
\[\left|X^{i}(\rho(\boldsymbol{x}))-X^{i+1}(\rho(\boldsymbol{x}))\right| \leq t\leq \frac{1}{2}.\]
Combining the fact that  for all $x \le 1/2$ and $k\in \mathbb{N}$, $\Psi^{(k)}(x) = 0$, we have $g_\mathcal{P}(\boldsymbol{x}') = g_\mathcal{P}^l(\boldsymbol{x}')$ with a probability at least $1-rd^{-\omega(1)}$.

By union bound over all queries $\boldsymbol{x}\in Q^l$, conditioned on that $\mathcal{E}_1,\ldots,\mathcal{E}_{l-1}$ occur,  with probability at least $1-r\mathsf{poly}(d)d^{-\omega(1)}$, $\mathcal{E}_l$ occurs.
Therefore by induction, 
\begin{align*}
    \mathbb{P}(\mathcal{E}_l) ~=~& \mathbb{P}(\mathcal{E}_l|\mathcal{E}_1,\ldots,\mathcal{E}_{l-1})\mathbb{P}(\mathcal{E}_{l-1}|\mathcal{E}_1,\ldots,\mathcal{E}_{l-2})\ldots \mathbb{P}(\mathcal{E}_{2}|\mathcal{E}_1) \mathbb{P}(\mathcal{E}_1)\\
    ~\geq~& 1-r^2\mathsf{poly}(d)d^{-\omega(1)} = 1-d^{-\omega(1)}.
\end{align*}
This implies that with high probability, the computation path in round $l$ is determined by $P_1,\ldots,P_{2l}$. Consequently,  for all $l\in [\tau]$ a solution returned after $l - 1$ rounds is determined by $P_1,\ldots,P_{2l-2}$ with high probability.
By the same concentration argument, for any $t \geq \tau+1$, we have $|X^{t}(\rho(\boldsymbol{x}))- \mathbb{E} [X^{t}(\rho(\boldsymbol{x}))]|<\frac{1}{4}$ with probability $1-d^{-\omega(1)}$ over $\mathcal{P}$. 

Finally, we note that by allowing the algorithm to use random bits, the results are
a convex combination of the bounds above, so the same high-probability bounds are satisfied.
\end{proof}
\subsection{Proof of Lemma~\ref{lmm:large_grad}}

\label{app:large_grad}

\begin{lemma}[\textbf{Small weighted partition implies large gradient norm}]
For any $\boldsymbol{x}\in \mathbb{R}^d$, if $|X^{r+1}(\rho(\boldsymbol{x}))- X^{r}(\rho(\boldsymbol{x}))|<1$, then $\left\|\nabla f_\mathcal{P}(\boldsymbol{x})\right\|\geq 0.08$.
\end{lemma}

\begin{proof}
We let $\boldsymbol{y} = \rho(\boldsymbol{x})$ and take $j\leq r+1$ to be the smallest $j$ for which $|X^{j}(\boldsymbol{y})- X^{j-1}(\boldsymbol{y})|<1$, so that $|X^{j-1}(\boldsymbol{y})- X^{j-2}(\boldsymbol{y})|\geq 1$, where we let $X^{0}\equiv 0$ and $X^{-1}\equiv 1$. 
We first show a slight stronger argument for the gradient of $g_\mathcal{P}$. Specifically, we will show there exits a direction vector $\boldsymbol{u}\in \mathbb{R}^d$ with $\left\|\boldsymbol{u}\right\|=1$ satisfies
  \[\left|\langle \boldsymbol{u},\boldsymbol{y}\rangle\right|<1\quad\mbox{and}\quad \left|\langle \boldsymbol{u},\nabla g_\mathcal{P}(\boldsymbol{y})\rangle\right|>\frac{1}{\sqrt{2}}.\]
To show it, we consider the $j=1$, $1<j<r+1$ and $j=r+1$ separately.
\paragraph{Case (i): $j=1$.} For $l \in P_1$, we have
\begin{align*}
&\sqrt{d_0}\partial_l g_\mathcal{P}(\boldsymbol{y})\\~=~&  -\Psi(1)\Phi'(X^1 ) 
- \Psi'(-X^1)\Phi(X^2-X^1) - \Psi(-X^1)\Phi'(X^2-X^1)\\
&-\Psi'(X^1)\Phi(X^1-X^2) - \Psi(X^1)\Phi'(X^1-X^2) \\
&  - \Psi'(X^1- X^2)\Phi(X^3 - X^2) 
- \Psi'(-X^1+ X^2)\Phi(X^2 - X^3)  \\
~\leq~&-\Psi(1)\Phi'(X^1).
\end{align*}
Now we choose the direction vector $\boldsymbol{u}$ as  $\boldsymbol{u}_l = \begin{cases}
\frac{1}{\sqrt{d_0}} & l\in P_1,\\
0 &\mbox{otherwise},
\end{cases}$ and combining the fact that for all $x \ge 1$ and $|y| < 1$,
    $\Psi(x)\Phi'(y) > 1$, we have $\left|\langle \boldsymbol{u},\boldsymbol{y}\rangle\right| = |X^1(\boldsymbol{y})|<1$ and 
\[\left|\langle \boldsymbol{u},\nabla g_\mathcal{P}(\boldsymbol{y})\rangle\right| = \left\|\frac{1}{\sqrt{d_0}}\sum\limits_{l\in P_1}\partial_l g_\mathcal{P}(\boldsymbol{y})\right\|\geq 1.\]
\paragraph{Case (ii): $1<j<r+1$.} For $l \in P_j$, we have
\begin{align*}
&\sqrt{d_0}\partial_l g_\mathcal{P}(\boldsymbol{y})\\
~=~&  -(-1)^{j+1}\left[\Psi'(X^{j}- X^{j+1})\Phi(X^{j+2}- X^{j+1}) 
 + \Psi'(X^{j+1}- X^{j})\Phi(X^{j+1} - X^{j+2}) \right]
 \\
 &-(-1)^{j} \left[- \Psi'(X^{j-1}- X^j)\Phi(X^{j+1}-X^{j})- \Psi(X^{j-1}- X^j)\Phi'(X^{j+1}-X^{j})\right.
 \\&
 \left.-\Psi'(X^j - X^{j-1})\Phi(X^j-X^{j+1}) -\Psi(X^j - X^{j-1})\Phi'(X^j-X^{j+1})\right] \\&
 -(-1)^{j-1}\left[\Psi(X^{j-2}- X^{j-1})\Phi'(X^{j}- X^{j-1}) + \Psi(X^{j-1}- X^{j-2})\Phi'(X^{j-1}-X^{j})\right]  \\
 ~=~& (-1)^{j} \left[\Psi'(X^{j}- X^{j+1})\Phi(X^{j+2}- X^{j+1}) 
 + \Psi'(X^{j+1}- X^{j})\Phi(X^{j+1} - X^{j+2}) \right.
 \\
 & + \Psi'(X^{j-1}- X^j)\Phi(X^{j+1}-X^{j})+ \Psi(X^{j-1}- X^j)\Phi'(X^{j+1}-X^{j})
 \\&
+\Psi'(X^j - X^{j-1})\Phi(X^j-X^{j+1}) +\Psi(X^j - X^{j-1})\Phi'(X^j-X^{j+1}) \\&
+\left.\Psi(X^{j-2}- X^{j-1})\Phi'(X^{j}- X^{j-1}) + \Psi(X^{j-1}- X^{j-2})\Phi'(X^{j-1}-X^{j})\right].
\end{align*}
In this case, combining the fact that $\Psi(x)= 0,~\forall x<1/2$ and  the fact that for all $x \ge 1$ and $|y| < 1$,
    $\Psi(x)\Phi'(y) > 1$,  we have $\mathrm{sign}(\partial_l g_\mathcal{P}(\boldsymbol{y})) = (-1)^j$ and 
\begin{align*}
\sqrt{d_0}|\partial_l g_\mathcal{P}(\boldsymbol{y})|~ \geq ~&|\Psi(X^{j-2}- X^{j-1})\Phi'(X^{j}- X^{j-1}) + \Psi(X^{j-1}- X^{j-2})\Phi'(X^{j-1}-X^{j})| \\
~ = ~&|\Psi(|X^{j-2}- X^{j-1}|)\Phi'((X^{j}- X^{j-1})\cdot \mathrm{sign}(X^{j-2}- X^{j-1}) )| \\
~ \geq  ~&1.
\end{align*}
Furthermore, for $l\in P_{j-1}$, we have $\mathrm{sign}(\partial_l g_\mathcal{P}(\boldsymbol{y})) = (-1)^{j-1}$. Thus if we choose the direction vector $\boldsymbol{u}$ as  $\boldsymbol{u}_l = \begin{cases}
\frac{1}{\sqrt{2d_0}} & l\in P_j,\\
-\frac{1}{\sqrt{2d_0}} & l\in P_{j-1},\\
0 &\mbox{otherwise},
\end{cases}$ we have $\left|\langle \boldsymbol{u},\boldsymbol{y}\rangle\right| = \frac{1}{\sqrt{2}}|X^j(\boldsymbol{y}) - X^{j-1}(\boldsymbol{y})|<1$ and 
\[\left|\langle \boldsymbol{u},\nabla g_\mathcal{P}(\boldsymbol{y})\rangle\right| \geq \left\|\frac{1}{\sqrt{2d_0}}\sum\limits_{l\in P_j}\partial_l g_\mathcal{P}(\boldsymbol{y})\right\|\geq \frac{1}{\sqrt{2}} .\]
\paragraph{Case (iii): $j=r+1$.} For $l \in P_{r+1}$, we have 
\begin{align*}
\sqrt{d_0}\partial_l g_\mathcal{P}(\boldsymbol{y})
~=~&  -(-1)^{r}\left[\Psi(X^{r-1}- X^{r})\Phi'(X^{r+1}- X^{r}) 
 + \Psi(X^{r}- X^{r-1})\Phi'(X^{r} - X^{r+1}) \right].
\end{align*}
Similarly, we have $\mathrm{sign}(\partial_l g_\mathcal{P}(\boldsymbol{y})) = (-1)^{r+1}$ and 
\[\sqrt{d_0}|\partial_l g_\mathcal{P}(\boldsymbol{y})| = |\Psi(X^{r-1}- X^{r})\Phi'(X^{r+1}- X^{r}) 
 + \Psi(X^{r}- X^{r-1})\Phi'(X^{r} - X^{r+1})|\geq 1.\]
Furthermore, for $l\in P_{r}$, we have $\mathrm{sign}(\partial_l g_\mathcal{P}(\boldsymbol{y})) = (-1)^{r}$. Thus if we choose the direction vector $\boldsymbol{u}$ as  $\boldsymbol{u}_l = \begin{cases}
\frac{1}{\sqrt{2d_0}} & l\in P_{r+1},\\
-\frac{1}{\sqrt{2d_0}} & l\in P_{r},\\
0 &\mbox{otherwise},
\end{cases}$ we have $\left|\langle \boldsymbol{u},\boldsymbol{y}\rangle\right| = \frac{1}{\sqrt{2}}|X^{r+1}(\boldsymbol{y}) - X^{r}(\boldsymbol{y})|<1$ and 
\[\left|\langle \boldsymbol{u},\nabla g_\mathcal{P}(\boldsymbol{y})\rangle\right| \geq \left\|\frac{1}{\sqrt{2d_0}}\sum\limits_{l\in P_{r+1}}\partial_l g_\mathcal{P}(\boldsymbol{y})\right\|\geq \frac{1}{\sqrt{2}} .\]

Next, we present the following lemma to address the unbounded domain, as discussed in Section 5.2 of \citep{carmon2020lower}, and provide its proof here for the reader’s convenience.
\begin{lemma}
Assume functions $g:\mathbb{R}^d \to \mathbb{R}$ with $\left\|\nabla g(\boldsymbol{y})\right\|\leq 46\sqrt{r+1}$ for any $\boldsymbol{y}\in \mathbb{R}^d$ and $\rho:\mathbb{R}^d \to \mathbb{R}^d$ defined as $\rho(\boldsymbol{x}) = \frac{\boldsymbol{x}}{\sqrt{1+\norm{\boldsymbol{x}}^2/R^2}} 
  ~\mbox{and}~R=230\sqrt{r+1}$. 
  For any point $\boldsymbol{x}\in \mathbb{R}^d $  such that $\rho(\boldsymbol{x}) = \boldsymbol{y}$  if there exits a direction vector $\boldsymbol{u}\in \mathbb{R}^d$ with $\left\|\boldsymbol{u}\right\|=1$ satisfies
  \[\left|\langle \boldsymbol{u},\boldsymbol{y}\rangle\right|<1\quad\mbox{and}\quad \left|\langle \boldsymbol{u},\nabla g(\boldsymbol{y})\rangle\right|\geq \frac{1}{\sqrt{2}},\]
  then 
  \[\left\|\nabla\big(g \circ \rho + \frac{1}{5}\left\|\cdot\right\|^2\big)(\boldsymbol{x}) \right\|\geq 0.08.\]
\end{lemma}

\begin{proof}
By $\frac{\partial \rho}{\partial \boldsymbol{x}}(\boldsymbol{x}) = \frac{I - \rho(\boldsymbol{x})
    \rho(\boldsymbol{x})^{\top}/R^2}{\sqrt{1+\norm{\boldsymbol{x}}^2/R^2}}$, we can calculate $\langle \boldsymbol{u},\nabla (g \circ \rho + \frac{1}{5}\left\|\cdot\right\|^2)(\boldsymbol{x})\rangle$ as
\begin{align*}
&\bigg\langle \boldsymbol{u},\nabla \big(g \circ \rho + \frac{1}{5}\left\|\cdot\right\|^2\big)(\boldsymbol{x})\bigg\rangle\\ ~=~&\bigg\langle \boldsymbol{u},\frac{\partial \rho}{\partial \boldsymbol{x}}(\boldsymbol{x})\nabla g(\rho(\boldsymbol{x})) \bigg\rangle + \frac{2}{5}\langle \boldsymbol{u},\boldsymbol{x}\rangle\\
~=~&\frac{\langle \boldsymbol{u}, \nabla g(\rho(\boldsymbol{x}))\rangle - \langle \boldsymbol{u}, \rho(\boldsymbol{x})\rangle \langle \rho(\boldsymbol{x}), \nabla g(\rho(\boldsymbol{x}))\rangle /R^2 }{\sqrt{1+\norm{\boldsymbol{x}}^2/R^2}} + \frac{2}{5}\langle \boldsymbol{u},\rho(\boldsymbol{x})\rangle \sqrt{1+\norm{\boldsymbol{x}}^2/R^2}.
\end{align*}
\paragraph{Case (i): $\left\|\boldsymbol{x}\right\|\leq R/2$.} In this case, since $\left\|\rho(\boldsymbol{x})\right\|\leq \left\|\boldsymbol{x}\right\|\leq R/2$, we have
\begin{align*}
&\left\|\nabla\big(g \circ \rho + \frac{1}{5}\left\|\cdot\right\|^2\big)(\boldsymbol{x}) \right\|\\
~\geq~&\left|\bigg\langle \boldsymbol{u},\nabla \big(g \circ \rho + \frac{1}{5}\left\|\cdot\right\|^2\big)(\boldsymbol{x})\bigg\rangle\right|\\
~\geq~& \frac{2}{\sqrt{5}}\left|\bigg\langle \boldsymbol{u},\nabla g(\rho(\boldsymbol{x}))\bigg\rangle\right| - \left|\langle \boldsymbol{u}, \rho(\boldsymbol{x})\rangle \right|\left(\frac{\left\|\nabla g(\rho(\boldsymbol{x}))\right\|}{2R}+\frac{1}{\sqrt{5}}\right) \\ 
~\geq~& \frac{2}{\sqrt{5}}\left|\bigg\langle \boldsymbol{u},\nabla g(\boldsymbol{y})\bigg\rangle\right| - \left|\langle \boldsymbol{u}, \boldsymbol{y}\rangle \right|\left(\frac{R/5}{2R}+\frac{1}{\sqrt{5}}\right) \\ 
~\geq~& \frac{\sqrt{2}}{\sqrt{5}}-\frac{1}{10}-\frac{1}{\sqrt{5}}\geq 0.08.
\end{align*}
\paragraph{Case (ii): $\left\|\boldsymbol{x}\right\|> R/2$.} In this case, we have $\left\|\frac{\partial \rho}{\partial \boldsymbol{x}}(\boldsymbol{x})\right\|_{\mathrm{op}}\leq \frac{1}{\sqrt{1+\norm{\boldsymbol{x}}^2/R^2}}\leq \frac{2}{\sqrt{5}}$ and that $\left\|\nabla g(\rho(\boldsymbol{x}))\right\|\leq R/5$. Thus we have 
\[\left\|\nabla\big(g \circ \rho + \frac{1}{5}\left\|\cdot\right\|^2\big)(\boldsymbol{x}) \right\|\geq \frac{2}{5}\left\|\boldsymbol{x}\right\| - \left\|\frac{\partial \rho}{\partial \boldsymbol{x}}(\boldsymbol{x})\right\|_{\mathrm{op}}\left\|\nabla g(\rho(\boldsymbol{x}))\right\|\geq \frac{R}{5}-\frac{2}{\sqrt{5}}\frac{R}{5}>0.02R \geq \sqrt{r+1}\geq 1.\]
\end{proof}

\end{proof}

\section{Missing details for constant dimensional lower bound}


\label{app:reduction}
We first discuss how to reduce the problem of finding $\epsilon$-stationary points to finding local minimum of a monotone path function defined over the grid, following \citep{bubeck2020trap}.



Specifically, we consider grid graph $G_{n,d} = (V_{n,d}, E_{n,d})$ as $V_{n,d} = \{0,\ldots,n\}^d$ and $E_{n,d} = \{(\boldsymbol{u},\boldsymbol{v})\in V_{n,d}\times V_{n,d}: \left\| \boldsymbol{u}-\boldsymbol{v}\right\|_1\leq 1 \}$.
A sequence of vertices $(\boldsymbol{v}^0, \ldots, \boldsymbol{v}^n)$ is called a \emph{monotone path} in $G_{n,d}$ if $\boldsymbol{v}^0 = \boldsymbol{0}$ and for every $1 \leq i \leq n$, $\boldsymbol{v}^i - \boldsymbol{v}^{i-1}$ equals $\boldsymbol{e}_j$ for some $j \in [d]$.
In other words, the path starts at the origin and progresses by incrementing exactly one coordinate at each step.
If $(\boldsymbol{v}^0, \ldots, \boldsymbol{v}^n)$ is a monotone path, we associate to it a \emph{monotone path function} $P: V_{n,d} \to \mathbb{R}$ defined as
\[
P(\boldsymbol{v}) =
\begin{cases} 
-\|\boldsymbol{v}\|_1 & \text{if } \boldsymbol{v} \in \{\boldsymbol{v}^0, \ldots, \boldsymbol{v}^n\}, \\
\|\boldsymbol{v}\|_1 & \text{otherwise}.
\end{cases}
\]



For convenience, we may sometimes refer to the path function $P$ and the path $(\boldsymbol{v}_0, \ldots, \boldsymbol{v}_n)$ interchangeably. For $i = 0, \ldots, n$, we write $P_{i,d}$ for the set $P^{-1}(-i)$.
We denote the set of all monotone path functions on $G_{n,d}$ by $\mathrm{F}_{n,d}$. It is clear that if $P \in \mathrm{F}_{n,d}$ then $P_{n,d}$ is the only local minimum of $P$ and hence the global minimum. 
%
For any $\boldsymbol{v}\in V_{n,d}$, we define
\[
\mathrm{square}(\boldsymbol{v}) = \bigotimes\limits_{s=1}^d \left[\frac{\boldsymbol{v}_s}{n + 1},\frac{\boldsymbol{v}_s+1}{n + 1}\right] \subset [0,1]^d.
\]

The following lemma allows to construct a smooth function that preserves the structure of any monotone path function.

%
\begin{lemma}[\textbf{\cite[Section 3]{vavasis1993black}}]
\label{lemma:vav_const}
Fix $\varepsilon>0$ and let $n(\varepsilon)$ depends on $\varepsilon$. For any $P \in \mathrm{F}_{n(\varepsilon),d}$, there exists a function $\widehat{P}:[0,1]^d\to\R$ with the following properties:
	\begin{enumerate}
		\item $\widehat{P}$ is smooth.
		\item $\widehat{P} = f_P + \ell$, where $\ell$ is a linear function, which does not depend on $P$, and 
		$\mathrm{supp}(f_P)\subset \bigcup\limits_{i=0}^n\mathrm{square}\left(P_i\right),$
        where $\mathrm{supp}(\phi) = \{\boldsymbol{x}: \phi(\boldsymbol{x})\neq 0\}$.
		\item If $\boldsymbol{x} \in [0,1]^d$ is an $\varepsilon$-stationary point of $\widehat{P}$ then $\boldsymbol{x} \in \mathrm{square}\left(P_{n(\varepsilon),d}\right)$.
		\item if $P' \in \mathrm{F}_{n(\varepsilon),d}$ is another function and for some $i = 0,...,n$, $(P'_{i-1,d},P'_{i,d},P'_{i+1,d}) = (P_{i-1,d},P_{i,d},P_{i+1,d})$. Then
$\widehat{P}'\vert_{\mathrm{square}\left(P_{i,d}\right)} = \widehat{P}\vert_{\mathrm{square}\left(P_{i,d}\right)}.$
\item Assume the length of path is $k$ then $n(\varepsilon)$ satisfies 
$-\varepsilon\cdot k/n(\varepsilon) +1/n(\varepsilon)^2\geq 0.$
	\end{enumerate}
\end{lemma}

In particular, because $k \leq dn(\varepsilon)$, then $n(\varepsilon) \leq 1/\sqrt{d\varepsilon}$.  We now restate the result in Section 6 of \citep{bubeck2020trap} that finding the minimum of $P$ is as hard as finding an $\varepsilon$-stationary point of $\widehat{P}$. 
To this end, we give formal definition of problem of finding a local minimum of a function defined on a graph.


%

%
%

\paragraph{Finding local minimum problem.} Given a graph $G = (V, E)$ and oracle access to a function $f : V \to \R$,  the goal is to find a vertex $v$ that is a local minimum, i.e., $f(v) \leq f(u)$ for all $(u, v) \in E$. To compare the difficulty of different learning problems, we leverage the \emph{round-preserving} reduction introduced in \citep{branzei2022query}.

\begin{definition}[\textbf{Round-preserving reduction}]
A reduction from oracle-based problem $\mathrm{P}_1$ to
oracle-based problem $\mathrm{P}_2$ is round-preserving if for any instance of problem $\mathrm{P}_1$ with oracle $\mathcal{O}_1$,
the instance of problem $\mathrm{P}_2$ with oracle $\mathcal{O}_2$ given by the reduction satisfies that
\begin{enumerate}
    \item  A solution of the $\mathrm{P}_1$ instance can be obtained from any solution of the $\mathrm{P}_2$ instance without
any more queries on $\mathcal{O}_1$.
\item Each query to $\mathcal{O}_2$ can be answered by a constant number of queries to $\mathcal{O}_1$ in one round.
\end{enumerate}
\end{definition}
The following lemma showed that finding the minimum of $P$ is as hard as finding an $\varepsilon$-stationary point of $\widehat{P}$.
\begin{lemma}[\textbf{Section 6 in \citep{bubeck2020trap}}]
\label{lemma:reduction_monotone}
There is a round-preserving reduction such that any instance of finding $\varepsilon$-stationary point of $\widehat{P}$ is reduced to the instance of finding local minimum of a monotone path function $P$. Specifically, for any algorithm, the complexity of finding $\varepsilon$-stationary point of $\widehat{P}$ is no smaller than $1/(2d+1)$ times the complexity of finding local minimum of $P \in F_{n(\varepsilon), d}$. 
\end{lemma}

\paragraph{Lower bound for monotone path functions.} By Lemma~\ref{lemma:reduction_monotone}, the remaining task is to find the number of required queries in each number for finding local minimum for monotone path functions, which is obtained by \citet{branzei2022query} for constant iteration case. 
Furthermore, we also find their construction also work for algorithm with $\Theta(\log(1/\varepsilon))$ iterations.  
We summarize this results in the following lemma.
\begin{lemma}[\textbf{Adaptive complexity of finding local minimum~\citep{branzei2022query}}]
\label{lemma:lower_bound_monotone_path}
Assume 
For any (possible randomized) algorithm running within $k$-round, and issuing $n^{\frac{d^{k+1}-d^k}{d^k-1}}/(20d\cdot k)$ queries per round fails to find local minimum for monotone path function with length at most $nd$ with probability at least $7/40$.
\end{lemma}
Combining Lemma \ref{lemma:reduction_monotone} and Lemma \ref{lemma:lower_bound_monotone_path} yields the desired result in Theorem \ref{the:main2}. We note $n(\varepsilon) \leq 1/\sqrt{d\varepsilon}$ here.

\section{Constant iteration gradient flow trapping }
\label{app:constant_alg}
In this section, we will prove Theorem~\ref{the:main3}. We first restate several useful definitions and facts in Appendix~\ref{app:tools}, then we propose our algorithm for unconstrained case in Appendix~\ref{app:alg}. Finally we analyze the algorithm in Appendix~\ref{app:analysis} and extend the result to cube constrained case in Appendix~\ref{app:constrained}

\begin{theorem}
For $\varepsilon>0$, $d,k = \Theta(1)$, there is a deterministic algorithm finding $\varepsilon$-stationary points, running within $k$-round, with
\begin{itemize}
    \item $C(d,k,L,\Delta)\cdot \varepsilon^{-\frac{{d-1}}{\left(\frac{2d}{d+1}\right)^k-1}-{d-1}}$ queries per iteration when $f$ is unconstrained,
    \item $C(d,k,L)\cdot  \varepsilon^{-\frac{{d-1}}{2\left(\left(\frac{2d}{d+1}\right)^k-1\right)}-\frac{d-1}{2}}$ queries per iteration when $f$ is constrained on $[0,1]^d$.
\end{itemize}
\end{theorem}

\subsection{Useful tools}
\label{app:tools}
Following \cite{bubeck2020trap,hollender2023computational}, we begin with some definitions and technical lemmas.
\begin{definition}[\bf{$\varepsilon$-unreachable}]
For $\boldsymbol{x}, \boldsymbol{y} \in \R^d$ and $\varepsilon>0$ we say that $\boldsymbol{y}$ is $\varepsilon$-unreachable from $\boldsymbol{x}$ if the following holds:
\[f(\boldsymbol{y}) > f(\boldsymbol{x}) - \varepsilon\left\|\boldsymbol{x} - \boldsymbol{y}\right\|_2.\]
\end{definition}

\begin{definition}[\bf{Full-dimensional hyperrectangle}]
A $k$-dimensional hyperrectangle $R $ in $\R^d$
is a set of the form $R = [a_1, b_1]\times \cdots \times[a_d, b_d]$,
where $a_i \leq b_i$ for all $i \in [d]$, and $|\{i \in [d] : a_i < b_i\}| = k$. When $k = d$, we also say that $R$ is full-dimensional.
\end{definition}

\begin{lemma}[\bf{Lemma 3 in \citet{hollender2023computational}}]
\label{lmm:trap}
Let $R$ be a full-dimensional hyperrectangle in $\R^d$ and $\boldsymbol{x}$ a point in $R$. For any $\varepsilon > 0$, if
all $\boldsymbol{y} \in \partial R$ are $\varepsilon$-unreachable from $\boldsymbol{x}$, then $R$ contains an $\varepsilon$-stationary point of $f$.
\end{lemma}
\begin{corollary}[\bf{Corollary 2 in \citet{hollender2023computational}}]
\label{lmm:stat}
Let $R = [a_1, b_1] \times \cdots \times [a_d, b_d]$ be a full-dimensional hyperrectangle in $\R^d$ and $\boldsymbol{x}$ a
point in $R$. For any $\varepsilon > 0$, if all $\boldsymbol{y} \in \partial R$ are $(\varepsilon/2)$-unreachable from $\boldsymbol{x}$, and $\max_i(b_i - a_i) \leq \frac{\varepsilon}{2\sqrt{d}L}$, then $\boldsymbol{x}$ is an $\varepsilon$-stationary point of $f$.
\end{corollary}

\begin{definition}[\bf{Nice $\delta$-net}]
Let $\delta > 0$ and let $R$ be a $k$-dimensional hyperrectangle in $\R^d$. A set of
points $S \subseteq R$ is a nice $\delta$-net of $R$ if for any face $F$ of $R$ it holds that $S \cap F$ is a $\delta$-net of $F$.
\end{definition}

The construction of nice $\delta$-net with reasonable size can be summarized in the following lemma.
\begin{lemma}[Lemma 9 in \cite{hollender2023computational}]
\label{lmm:net_construct}
Let $R = [a_1,b_1] \times \dots \times [a_d,b_d]$ be a $k$-dimensional hyperrectangle in $\mathbb{R}^d$. Then, for any $\delta > 0$, there exists a nice $\delta$-net $S$ of $R$ with $|S| = \prod_{i=1}^d (\lceil \sqrt{k}(b_i-a_i)/2\delta \rceil + 1)$. In particular, if $R$ is $(d-1)$-dimensional, then we have $|S| \leq (\sqrt{d}r/2\delta)^{d-1}$ for any $r$ satisfying $r \geq \max_i (b_i-a_i)$ and $r \geq 8 \sqrt{d} \delta$.
\end{lemma}

The following lemma shows the unreachablility of $\delta$-net of $E$ implies unreachablility of $S$.

\begin{lemma}[Lemma 10 in \cite{hollender2023computational}]\label{lem:increase-eps}
Let $E$ be a $(d-1)$-dimensional hyperrectangle of $\mathbb{R}^d$ and let $S$ be a nice $\delta$-net of $E$. Let $\varepsilon > 0$ and let $\boldsymbol{x} \in \mathbb{R}^d$ with $\mathrm{dist}(\boldsymbol{x},E) > 0$. Then, if all $\boldsymbol{z} \in S$ are $\varepsilon$-unreachable from $\boldsymbol{x}$, it follows that all $\boldsymbol{y} \in E$ are $\varepsilon'$-unreachable from $\boldsymbol{x}$, where
\[\varepsilon' = \varepsilon + \frac{\delta^2}{2 \mathrm{dist}(\boldsymbol{x},E)} \left(L + \frac{2\varepsilon}{\mathrm{dist}(\boldsymbol{x},E)}\right).\]
\end{lemma}

\subsection{Gradient flow fully parallel trapping}
\label{app:alg}

Our constant iteration algorithm for finding stationary points is summarized in Algorithm~\ref{alg:main}.
In Line 2, the algorithm initializes the region with a unreachable boundary.  
In Lines 3--10, the gradient is trapped using multiple \( (d-1) \)-dimensional grids, followed by an update to the current solution.  
In Lines 11--18, the hyperrectangle is updated by reducing the length of its longest side to \( 1 / \ell_t \).  
In Line 20, the accuracy constant is slightly increased to ensure that the boundary of the updated hyperrectangle remains unreachable to the current solution.

\begin{algorithm2e}[H]
 \caption{The Gradient Flow Grid Trapping (GFGT) Algorithm
 }
  \label{alg:main}
  \SetAlgoLined
\DontPrintSemicolon
\SetNoFillComment
  \SetKwInOut{Input}{Input}
  \Input{
accuracy $\varepsilon > 0$, smooth parameter $L > 0$, dimension $d \geq 2$, initialization point $\boldsymbol{x}^0 \in \R^d$, query access to $f : \R^d \to  [0, +\infty)$ with $L$-Lipschitz $\nabla f$.
}
Set $t=0$ and $\varepsilon_0 = \varepsilon/4$,\;
Initialize hyperrectangle $R_0 :=[a_1^1,b_1^1]\times \cdots\times [a_d^1,b_d^1]$ as 
\[R_0 = \{\boldsymbol{x}\in \R^d:\left\|\boldsymbol{x}-\boldsymbol{x}^0\right\|_\infty \leq 2f(\boldsymbol{x}^0)/\varepsilon_0\}.\]
\For{$t+1 \in [k]$}{
Set $r_j^t = b_j^t- a_j^t$ for $j\in [d]$ and $r^t = \min_j r_j^t$,\;
Set the number of trap barriers on each coordinate direction $\ell_t = \left( \frac{4f(\boldsymbol{x}^0)\sqrt{d}L}{3^k\cdot\varepsilon_0\cdot \varepsilon}
\right)^{\frac{\frac{d-1}{d+1}\left(\frac{2d}{d+1}\right)^t}{\left(\frac{2d}{d+1}\right)^k-1}}$ ,\;
Set the gap size of trap barrier as $\delta_t = 
 \sqrt{\frac{\varepsilon r^t(3/4)^{\lfloor t d\rfloor}}{40\cdot 3^k\cdot \ell_t d\sqrt{d}L }}$, \;
Query $f$ on the nice $\delta_t$-net $S(E)$ of any $E\in \mathcal{E}^t$ where 
\[\mathcal{E}^t =\{[a_1^t,b_1^t]\times\cdots\times [a_j^t+mr_j^t/\ell_{t}]\times\cdots\times [a_d^t,b_d^t]: j\in [d], m\in [\ell_t-1]\} ,\]
Let $S^\star = \left\{\boldsymbol{z}\in \bigcup\limits_{ E\in \mathcal{E}^t}S(E) : f(\boldsymbol{z})\leq f(\boldsymbol{x}^t) - \varepsilon_t \left\|\boldsymbol{x}^t - \boldsymbol{z}\right\|_2\right\}$,\;
\uIf{$S^\star=\emptyset$}{$\boldsymbol{x}^{t+1}=\boldsymbol{x}^t$,}
\uElse{
$\boldsymbol{x}^{t+1}=\arg\min_{\boldsymbol{z}\in S^\star} f(\boldsymbol{z})$,
}
\For{$j\in [d]$}{
\uIf{ $a_j^{t}\leq [\boldsymbol{x}^{t+1}]_j\leq a_j^{t} + r_j^t/\ell_t $}{
$a_j^{t+1} = a_j^{t}$ and $b_j^{t+1} = a_j^{t} + 2r_j^t/\ell_t$, 
}
\uElseIf{$b_j^{t} - r_j^t/\ell_t\leq \boldsymbol{x}^{t+1}_j\leq b_j^{t}$}{
$a_j^{t+1} = b_j^{t}-2 r_j^t/\ell_t$ and $b_j^{t+1} = b_j^{t}$,
}\uElse{
Assume $a_j^{t}+ m_1r_j^t/\ell_t  \leq \boldsymbol{x}^{t+1}_j<a_j^{t} + m_2r_j^t/\ell_t $ with $m_1\geq 1$ and $m_2\leq \ell_t-1$,\;
update $a_j^{t+1} = a_j^{t} + (m_1-1) r_j^t/\ell_t$ and $b_j^{t+1} = a_j^{t}  + (m_2+1) r_j^t /\ell_t$,
}
}
Update $\varepsilon_{t+1} = \varepsilon_t + \frac{\varepsilon(3/4)^{\lfloor t d\rfloor}}{16d}$.\;
}
\Return ${\boldsymbol{x}_{k}}$.
\end{algorithm2e}

\subsection{Analysis of Algorithm \ref{alg:main}: Proof of the first part of Theorem \ref{the:main3}}
\label{app:analysis}

\begin{lemma}[Complexity]
For any $d\geq 2$, the number of queries to $f$ in each iteration is bound by
\[C(d,k,L) f(\boldsymbol{x}^0)^{(d-1)/2}\varepsilon^{-\frac{\frac{d-1}{2}}{\left(\frac{2d}{d+1}\right)^k-1}-{d-1}},\]
where $C(d,k,L)$ is a constant depends on $d,k,L$.
\end{lemma}

\begin{proof}
In iteration $t$, there is $d\cdot \ell_{t}$ hyperrectangles with dimension $d-1$ and  side length bounded by 
\[\left(\prod_{i=0}^{t-1}\frac{3}{\ell_t}\right) \cdot \frac{2f(\boldsymbol{x}^0)}{\varepsilon}.\]
Combining Lemma \ref{lmm:net_construct}, we have 
\begin{align*}
\# \mbox{queries in iteration }t &\leq \left( \sqrt{d}\cdot 3^t r^t \cdot \frac{1}{2 \sqrt{\frac{\varepsilon r^t(3/4)^{\lfloor t d\rfloor}}{40\cdot 3^k\cdot \ell_t d\sqrt{d}L }}} \right)^{d-1}\cdot d\cdot \ell_{t}\\
&\leq d\left( \sqrt{d}\cdot 3^t \cdot \frac{\sqrt{{40\cdot 3^k\cdot  d\sqrt{d}L }}}{2 \sqrt{ (3/4)^{\lfloor t d\rfloor}}} \right)^{d-1} \cdot\ell_t^{(d+1)/2}\cdot (r^t)^{(d-1)/2} \cdot\varepsilon^{-(d-1)/2}\\ 
&:= C(d,k,L)\cdot\ell_t^{(d+1)/2}\cdot (r^t)^{(d-1)/2}\cdot \varepsilon^{-(d-1)/2}\\ 
&\leq C(d,k,L)\cdot\ell_t^{(d+1)/2}\cdot \left(\left(\prod_{i=0}^{t-1}\frac{3}{\ell_t}\right) \cdot \frac{2f(\boldsymbol{x}^0)}{\varepsilon_0}\right)^{(d-1)/2}\cdot \varepsilon^{-(d-1)/2}\\ 
&\leq C(d,k,L)\cdot\left(\frac{2f(\boldsymbol{x}^0)}{\varepsilon_0}\right)^{(d-1)/2}\cdot \varepsilon^{-\frac{{d-1}}{\left(\frac{2d}{d+1}\right)^k-1}-\frac{d-1}{2}}
\end{align*}
\end{proof}

\begin{lemma}[\bf{Correctness}]
The output of the algorithm $\boldsymbol{x}^k$ is an $\varepsilon$-stationary point of $f$.
\end{lemma}

\begin{proof}
We assume $f(\boldsymbol{x}^0)>0$ otherwise, $\boldsymbol{x}^0$ is a global minimum and thus a stationary point.

Let $R_t$ denote the hyperrectangle at the beginning of iteration $t = 0,\ldots,k-1$. We consider the following invariant at every iteration: for all $\boldsymbol{y}\in \partial R_t$ are $\varepsilon_t$-unreachable from $\boldsymbol{x}^t$.

If such invariant holds, we observe that 
\[\max_i(b_i^k- a_i^k)\leq \left(\prod_{i=0}^{k-1}\frac{3}{\ell_t}\right) \cdot \frac{2f(\boldsymbol{x}^0)}{\varepsilon_0}\leq \frac{\varepsilon}{2\sqrt{d}L},
\]
and $\varepsilon_{k}$ can be bounded as 
\[\varepsilon_{k} = \varepsilon_0 + \sum\limits_{t=0}^{T-1}\frac{\varepsilon(3/4)^{\lfloor t d\rfloor}}{16d}\leq \varepsilon_0 + \sum\limits_{t=0}^{\infty}\frac{\varepsilon(3/4)^{\lfloor t d\rfloor}}{16d} \leq \frac{\varepsilon}{2}.\]
By Corollary \ref{lmm:stat}, $\boldsymbol{x}_{k}$ is $\varepsilon$-stationary point of $f$.

Now we prove the invariant by induction. For $t=0$, for any $\boldsymbol{y}\in R_0$, by definition of $R_0$ we have $\left\|\boldsymbol{x}^0-\boldsymbol{y}\right\|_2\geq \left\|\boldsymbol{x}^0-\boldsymbol{y}\right\|_\infty = 2f(\boldsymbol{x}^0)/\varepsilon_0$ thus 
\[f(\boldsymbol{y})\geq 0\geq 2f(\boldsymbol{x}^0) - \varepsilon_0 \left\|\boldsymbol{x}^0-\boldsymbol{y}\right\|_2 > f(\boldsymbol{x}^0) - \varepsilon_0 \left\|\boldsymbol{x}^0-\boldsymbol{y}\right\|_2.\]
Now we show the invariant holds for iteration $t+1$ conditioned on that it holds for iteration $t$.
\paragraph{Case 1: $S^\star = \emptyset$.} In this case $\boldsymbol{x}^{t+1}  = \boldsymbol{x}^t$. For $\boldsymbol{y}\in \partial R_{t+1}\cap \partial R_t$ are $\varepsilon_t$-unreachable from $\boldsymbol{x}^t$.
Thus by $\varepsilon_{t+1}\geq \varepsilon_t$, it is also $\varepsilon_{t+1}$-unreachable from $\boldsymbol{x}^{t+1}$. For $\boldsymbol{y}\in \partial R_{t+1}\setminus \partial R_t\subseteq \bigcup\limits_{ E\in \mathcal{E}^t\cap \partial R_{t+1}}S(E)$, we show as follows.
By construction in Lines 12--19, we have 
\[\mathrm{dist}(\boldsymbol{x}^t,E)\geq \min_j{r_j^t}/\ell_t\geq \left(\prod_{i=0}^{k-1}\frac{1}{\ell_t}\right) \cdot \frac{2f(\boldsymbol{x}^0)}{\varepsilon_0} = \frac{1}{3^k}\frac{\varepsilon}{2\sqrt{d}L}\] 
Combining $S^\star = \emptyset$ and Lemma \ref{lem:increase-eps}, we can claim that for all $\boldsymbol{y} \in \bigcup\limits_{ E\in \mathcal{E}^t\cap \partial R_{t+1}}S(E)$ are $\varepsilon'$-unreachable from $\boldsymbol{x}^{t+1}$, where 
\begin{align*}
\varepsilon' &=
\varepsilon_t + \frac{\delta^2}{2 \mathrm{dist}(\boldsymbol{x},E)} \left(L + \frac{2\varepsilon}{\mathrm{dist}(\boldsymbol{x},E)}\right)\\
&\leq 
\varepsilon_t + \frac{\delta^2}{2 \mathrm{dist}(\boldsymbol{x},E)} \left(L + 3^k 4\sqrt{d}L\right)\\
&\leq 
\varepsilon_t + \frac{\delta^2}{2 \mathrm{dist}(\boldsymbol{x},E)} 3^k 5\sqrt{d}L\\
&\leq 
\varepsilon_t + \frac{\delta^2 \ell_t}{2 r^t} 3^k 5\sqrt{d}L \\
& \leq  
\varepsilon_t + \frac{\varepsilon(3/4)^{\lfloor t d\rfloor}}{16d}  = \varepsilon_{t+1}.
\end{align*}
\paragraph{Case 2: $S^\star \neq \emptyset$.} In this case, $\boldsymbol{x}^{t+1}=\arg\min_{\boldsymbol{z}\in S^\star} f(\boldsymbol{z})$.
For $\boldsymbol{y}\in \partial R_{t+1}\cap \partial R_t$, we have 
\[f(\boldsymbol{y})\geq f(\boldsymbol{x}^t)- \varepsilon_t \left\|\boldsymbol{x}^t- \boldsymbol{y}\right\|_2 \geq f(\boldsymbol{x}^{t+1}) +\varepsilon_t\left\|\boldsymbol{x}^t- \boldsymbol{x}^{t+1}\right\|_2 - \varepsilon_t\left\|\boldsymbol{x}^t- \boldsymbol{y}\right\|_2 \geq f(\boldsymbol{x}^{t+1})  - \varepsilon_t\left\|\boldsymbol{x}^{t+1}- \boldsymbol{y}\right\|_2,\]
\emph{i.e. }, $\boldsymbol{y}$ is $\varepsilon_t$-unreachable from $\boldsymbol{x}^{t+1}$. Since $\varepsilon_{t+1}\geq \varepsilon_t$, we have $\boldsymbol{y}$ is $\varepsilon_{t+1}$-unreachable from $\boldsymbol{x}^{t+1}$.
For $\boldsymbol{y}\in \partial R_{t+1}\setminus \partial R_t$, we first show for every $\boldsymbol{y}\in S(E)$ where $E\in \mathcal{E}^t\cap \partial R_{t+1}$, it is $\varepsilon_t$-unreachable for $\boldsymbol{x}^{t+1}$. We note $\partial R_{t+1}\setminus \partial R_t\subseteq \mathcal{E}^t$. 
On the one hand, for $\boldsymbol{y}\in S(E)\setminus S^\star$, it is $\varepsilon_t$-unreachable from $\boldsymbol{x}^{t}$. Thus by same argument before, it is $\varepsilon_t$-unreachable from $\boldsymbol{x}^{t+1}$.
On the other hand, for $\boldsymbol{y}\in S(E)\cap S^\star$, since $\boldsymbol{x}^{t+1}=\arg\min_{\boldsymbol{z}\in S^\star} f(\boldsymbol{z})$ and $\left\|\boldsymbol{x}^{t+1} - \boldsymbol{y}\right\|_2>0$, we have
\[f(\boldsymbol{y})> f(\boldsymbol{x}^{t+1}) - \varepsilon_t \left\|\boldsymbol{x}^{t+1} - \boldsymbol{y}\right\|_2,\]
\emph{i.e. }, $\boldsymbol{y}$ is $\varepsilon_t$-unreachable from $\boldsymbol{x}^{t+1}$.
Finally, we use Lemma \ref{lem:increase-eps} as before and  claim that for all $\boldsymbol{y} \in \bigcup\limits_{ E\in \mathcal{E}^t\cap \partial R_{t+1}}S(E)$ are $\varepsilon_{t+1}$-unreachable from $\boldsymbol{x}^{t+1}$.
\end{proof}

\subsection{Constrained setting: proof of second part of Theorem \ref{the:main3}}
\label{app:constrained}
In this section, we adapt Algorithm~\ref{alg:main} to the problem of finding stationary points constrained on cube $[0,1]^d$. The goal is to find an $\varepsilon$-stationary point, also known as an $\varepsilon$-KKT point, \emph{i.e.}, a point $\boldsymbol{x}\in [0,1]^d$ such that $\left\|g(\boldsymbol{x})\right\|_2\leq\varepsilon$, where $g$ is the projected gradient of $f$ on $[0,1]^d$ defined as 
\[g_i(\boldsymbol{x}) = \begin{cases}
\min \{0,[\nabla f(\boldsymbol{x})]_i \} &\mbox{if } x_i = 0\\
[\nabla f(\boldsymbol{x})]_i & \mbox{if }x_i \in  (0,1)\\
\max \{0,[\nabla f(\boldsymbol{x})]_i \} &\mbox{if } x_i = 1.
\end{cases}\]

We modify the algorithm as follows. We also note similar argument was proposed in \cite[Section 4.3]{hollender2023computational}. 
\begin{itemize}
    \item \textbf{Initialization:} We initialize $\boldsymbol{x}^0$ to be an arbitrary point in $[0,1]^d$ instead of 
    \[R_0 = \{\boldsymbol{x}\in \R^d:\left\|\boldsymbol{x}-\boldsymbol{x}^0\right\|_\infty \leq 2f(\boldsymbol{x}^0)/\varepsilon_0\}.\]
    \item \textbf{The number of barriers.} We set $r^0 =1$ instead of $r^0 = 2f(\boldsymbol{x}^0)/\varepsilon_0$, and 
    \[\ell_t = \left( \frac{2\sqrt{d}L}{3^k\cdot \varepsilon}
\right)^{\frac{\frac{d-1}{d+1}\left(\frac{2d}{d+1}\right)^t}{\left(\frac{2d}{d+1}\right)^k-1}}.\]
With this setting, we have 
\[\max_i(b_i^k- a_i^k)\leq \left(\prod_{i=0}^{k-1}\frac{3}{\ell_t}\right)\leq \frac{\varepsilon}{2\sqrt{d}L},
\]
    \item \textbf{Extraction of a solution:} Instead of outputting $\boldsymbol{x}_d$, we check the $2^d =\Theta(1)$ corners of $R_k$ until we find one with $\left\|g(\boldsymbol{y})\right\|_2 \leq \varepsilon$. We can use $O(d)$ queries to $f$ to compute a sufficiently good approximation of $\nabla f$ and thus $g$ at any point.
\end{itemize}
The rest of the algorithm is unchanged. The analysis is also almost same and we highlight the main difference as follows.
\begin{itemize}
    \item \textbf{Correctness:} The invariant is modified as: for all $\boldsymbol{y}\in \partial R_t\setminus \partial[0,1]^d$ are $\varepsilon_t$-unreachable from $\boldsymbol{x}^t$. The proof of this new invariant is identical to the old one.

    This new invariant, combined with a modified version of the proof of Lemma~\ref{lmm:trap} (where we analyze the piecewise differentiable \emph{projected} gradient flow defined by $\gamma'(t) = -g(\gamma(t))$), establishes the following: there exists a point $\boldsymbol{x} \in R_T$ such that $\left\|g(\boldsymbol{x})\right\|_2 \leq \eps/2$. A straightforward argument then shows that any corner $\boldsymbol{y}$ of the face of $R_k$ containing $\boldsymbol{x}$ must satisfy $\left\|g(\boldsymbol{y})\right\|_2 \leq \eps$. Hence, the algorithm successfully returns an $\eps$-stationary point.
    \item \textbf{Running time:} The number of queries used by algorithm is identical to the unconstrained version, namely,  
    \begin{align*}
\# \mbox{queries in iteration }t 
&\leq \left( \sqrt{d}\cdot 3^t r^t \cdot \frac{1}{2 \sqrt{\frac{\varepsilon r^t(3/4)^{\lfloor t d\rfloor}}{40\cdot 3^k\cdot \ell_t d\sqrt{d}L }}} \right)^{d-1}\cdot d\cdot \ell_{t}\\
&\leq d\left( \sqrt{d}\cdot 3^t \cdot \frac{\sqrt{{40\cdot 3^k\cdot  d\sqrt{d}L }}}{2 \sqrt{ (3/4)^{\lfloor t d\rfloor}}} \right)^{d-1} \cdot\ell_t^{(d+1)/2}\cdot (r^t)^{(d-1)/2} \cdot\varepsilon^{-(d-1)/2}\\ 
&:= C(d,k,L)\cdot\ell_t^{(d+1)/2}\cdot (r^t)^{(d-1)/2}\cdot \varepsilon^{-(d-1)/2}\\ 
&\leq C(d,k,L)\cdot\ell_t^{(d+1)/2}\cdot \left(\left(\prod_{i=0}^{t-1}\frac{3}{\ell_t}\right) \right)^{(d-1)/2}\cdot \varepsilon^{-(d-1)/2}\\ 
&\leq C(d,k,L)\cdot \varepsilon^{-\frac{{d-1}}{\left(\frac{2d}{d+1}\right)^k-1}-\frac{d-1}{2}}.
\end{align*}
\end{itemize}

\end{document}